\documentclass[11pt,letterpaper]{amsart}

\usepackage[matrix,arrow,curve,frame]{xy}
\usepackage{amsmath,amsthm,amssymb,enumerate}
\usepackage{latexsym}
\usepackage{amscd}
\usepackage[colorlinks=false]{hyperref}
\usepackage{euscript}

\newtheorem{thm}[subsection]{Theorem}
\newtheorem{defn}[subsection]{Definition}
\newtheorem{prop}[subsection]{Proposition}

\newtheorem{cor}[subsection]{Corollary}
\newtheorem{lemma}[subsection]{Lemma}

\newtheorem{remark}[subsection]{Remark}

\theoremstyle{definition}

\numberwithin{equation}{section}

\def\tpartial{{\partial_*}}
\def\bpartial{{\bar\partial}}

\def\cA{\mathcal A}
\def\cB{\mathcal B}

\def\cD{\mathcal D}
\def\cE{\mathcal E}
\def\cF{\mathcal F}

\def\cI{\mathcal I}
\def\cJ{\mathcal J}

\def\cM{\mathcal M}

\def\cR{\mathcal R}
\def\cS{\mathcal S}
\def\cT{\mathcal T}

\def\cW{\mathcal W}
\def\R{\mathfrak R}

\def\B{\mathfrak B}

\def\Zplus{\mathbb Z_{\geq 0}}

\DeclareMathOperator{\sign}{sign}

\DeclareMathOperator{\Hom}{Hom}

\DeclareMathOperator{\Spec}{Spec}

\newfont{\german}{eufm10}

\begin{document}
\pagestyle{plain}

\title{Standard monomials and invariant theory for arc spaces I: general linear group}

\author{Andrew R. Linshaw}
\address{Department of Mathematics, University of Denver}
\email{andrew.linshaw@du.edu}

\author{Bailin Song}
\address{School of Mathematical Sciences, University of Science and Technology of China, Hefei, Anhui 230026, P.R. China}
\email{bailinso@ustc.edu.cn}

\begin{abstract}
This is the first in a series of papers on standard monomial theory and invariant theory of arc spaces. For any algebraically closed field $K$, we construct a standard monomial basis for the arc space of the determinantal variety over $K$. As an application, we prove the arc space analogue of the first and second fundamental theorems of invariant theory for the general linear group. 
\end{abstract}


\maketitle
\section{Introduction}

\subsection{Classical invariant theory} Classical invariant theory has a long history that began in the 19th century in work of Cayley, Gordan, Klein, and Hilbert. Given an algebraically closed field $K$, a reductive algebraic group $G$ over $K$, and and a finite-dimensional $G$-module $W$, the ring of invariant polynomial functions $K[W]^G$ is the main object of study. It is often useful to consider invariant rings $K[V]^G$, where $V=W^{\oplus p}\bigoplus {W^*}^{\oplus q}$ is the direct sum of $p$ copies of $W$ and $q$ copies of the dual $G$-module $W^*$. In the terminology of Weyl, a {\it first fundamental theorem of invariant theory} (FFT) for the pair $(G,W)$ is a generating set for $K[V]^G$, and a {\it second fundamental theorem} (SFT) for $(G,W)$ is a generating set for the ideal of relations among the generators of $K[V]^G$. When $\text{char} \ K = 0$, if $G$ is one of the classical groups and $W$ is the standard representation, the FFTs and SFTs are due to Weyl 
\cite{W}. The analogous results in arbitrary characteristic were proven by de Concini and Procesi in \cite{DCP}. Explicit FFTs and SFTs are in general difficult to obtain and are known only in a few other cases, such as the adjoint representations of the classical groups which is due to Procesi \cite{P}, the $7$-dimensional representation of $G_2$ and the $8$-dimensional representation of $\text{Spin}_7$, which are due to Schwarz \cite{Sc}.

 The main example in this paper is the case where $G$ is the general linear group $GL_h(K)$ over $K$, and $W=K^{\oplus h}$ is its standard representation. For $V = W^{\oplus p}\bigoplus {W^*}^{\oplus q}$ as above, the affine coordinate ring is $$K[V]=K[a^{(0)}_{il}, b^{(0)}_{jl}|\ 1\leq i\leq p, \ 1\leq j\leq q,\ 1\leq l\leq h].$$

\begin{thm}\label{thm:classicalmain} (FFT and SFT for $G = GL_h(K)$ and $W= K^{\oplus h}$)
	\begin{enumerate}
		\item The ring of invariants $K[V]^{GL_h(K)}$ is generated by 
		$$\{X^{(0)}_{ij} = \sum_l a^{(0)}_{il}b^{(0)}_{jl}|\ 1\leq i\leq p, \ 1\leq j\leq q\}.$$
		\item The ideal of relations among the generators in (1) is generated by 
		$$ \left|
		\begin{array}{cccc}
		X^{(0)}_{u_1v_1} & X^{(0)}_{u_1v_2} & \cdots &X^{(0)}_{u_1v_{h+1}}\\
		X^{(0)}_{u_2v_1}& X^{(0)}_{u_2v_2} & \cdots & X^{(0)}_{u_2v_{h+1}} \\
		\vdots & \vdots & \vdots & \vdots \\
		X^{(0)}_{u_{h+1}v_1} & X^{(0)}_{u_{h+1}v_2} & \cdots & X^{(0)}_{u_{h+1}v_{h+1}} \\
		\end{array}
		\right|,
		$$ for all $u_1, u_2,\dots, u_h$ and $v_1,v_2,\dots, v_h$ with $1\leq u_i < u_{i+1} \leq p$ and $1\leq v_i < v_{i+1} \leq q$.
		\end{enumerate}
		\end{thm}

\subsection{Standard monomial theory} Standard monomial theory was initiated in the 1970s by Seshadri, Musili and Lakshmibai \cite{Se, LS,LMS1,LMS2}, generalizing earlier work of Hodge \cite{H}. It involves nice combinatorial bases for the affine coordinate rings of Schubert varieties inside quotients of classical groups by parabolic subgroups. In this paper, we only need the case of determinantal varieties.

For positive integers $p$ and $q$, let 
\begin{equation} \label{def:ringR} R= R_{p,q} = \mathbb Z[x^{(0)}_{ij}|\ 1\leq i\leq p, \ 1\leq j\leq q],\end{equation} be the ring of polynomial functions with integer coefficients on the space of $p\times q$ matrices. Consider the $h$-minor 
\begin{equation}\label{eqn:minor} B=\left|
\begin{array}{cccc}
x^{(0)}_{u_1v_1} & x^{(0)}_{u_1v_2} & \cdots &x^{(0)}_{u_1v_h}\\
x^{(0)}_{u_2v_1}& x^{(0)}_{u_2v_2} & \cdots & x^{(0)}_{u_2v_h} \\
\vdots & \vdots & \vdots & \vdots \\
x^{(0)}_{u_hv_1} & x^{(0)}_{u_hv_2} & \cdots & x^{(0)}_{u_hv_h} \\
\end{array}
\right|,
\end{equation}
with $u_i<u_{i+1}$, $v_i<v_{i+1}$. Throughout this paper, we will represent $B$ by the pair of ordered $h$-tuples
$$(u_h,\dots,u_2,u_1|v_1,v_2,\dots,v_h).$$ There is a partial ordering on the set of these minors given by
\begin{equation*} \begin{split} & (u_h,\dots,u_2,u_1|v_1,v_2,\dots,v_h)\leq (u'_{h'},\dots,u'_2,u'_1|v'_1,v'_2,\dots,v'_{h'}), 
\\ & \text{if } \ h'\leq h, \ u_i\leq u'_i, \ v_i\leq v'_i.\end{split} \end{equation*}
$R$ has a standard monomial basis (cf.~\cite{LR}) with
respect to this partially ordered set of minors: the ordered products $A_1A_2\cdots A_k$ of minors $A_i$ with $A_i\leq A_{i+1}$, form a basis of $R$. Similarly, let $R[h]$ be the ideal of $R$ generated by all $h$-minors in the form of (\ref{eqn:minor}), and let 
\begin{equation} \label{def:ringR_h} R_h=R\slash R[h+1].\end{equation} Then $R_h$ has a basis consisting of ordered products $A_1A_2\cdots A_k$ of $h_i$-minors $A_i$ with $h_i < h$ and $A_i\leq A_{i+1}$.

 For an arbitrary algebraically closed field $K$, let $M_{p,q} = M_{p,q}(K)$ be the space of $p\times q$ matrices with entries in $K$. The affine coordinate ring  $K[M_{p,q}]$ is obtained from $R$ by base change, that is, $K[M_{p,q}]\cong R\otimes_{\mathbb{Z}} K$. Let $K[M_{p,q}][h]$ be the ideal generated by all $h$-minors. The determinantal variety $D_h = D_h(K)$ is a closed subvariety of $M_{p,q}$ with $K[M_{p,q}][h]$ as the defining ideal. Then the affine coordinate ring $K[D_h] \cong K[M_{p,q}]\slash K[M_{p,q}][h]$, has a standard monomial basis: the ordered products $A_1A_2\cdots A_k$ of $h_i$-minors $A_i$ with $h_i<h$ and $A_i\leq A_{i+1}$ form a basis of $K[D_h]$. With $G = GL_h(K)$ and $V$ as in Theorem \ref{thm:classicalmain}, we have $V/\!\!/G \cong D_{h+1}$, and the proof of Theorem \ref{thm:classicalmain} in \cite{DCP} makes use of this standard monomial basis. A uniform treatment of the FFT and SFT for all the classical groups using standard monomial theory can also be found in the book \cite{LR}.

 \subsection{Arc spaces} 
 
 For a scheme $X$ of finite type over $K$, the arc space $J_\infty(X)$ is defined as the inverse limit of the finite jet schemes $J_n(X)$ \cite{EM}. By Corollary 1.2 of \cite{B}, it is determined by its functor of points : for every $K$-algebra $A$, we have a bijection
$$\Hom(\Spec  A, J_\infty(X))\cong\Hom(\Spec A[[t]], X).$$ If $i: X\to Y$ is a morphism of schemes, we get a morphism of schemes $i_{\infty}:J_{\infty}(X)\to J_{\infty}(Y)$. Arc spaces were first studied by Nash in \cite{Na}, and carry important information about the singularities of $X$. The {\it Nash problem} asks whether there is a bijection between the irreducible components of $J_{\infty}(X)$ lying over the singular locus of $X$, and the essential divisors over $X$. This has been answered affirmatively for many classes of varieties, although counterexamples are known \cite{IK}. Arc spaces are also important in Kontsevich's theory of {\it motivic integration}, which was used to prove that birationally equivalent Calabi-Yau manifolds have the same Hodge numbers \cite{Kon}. This theory has been developed by many authors including Batyrev, Craw, Denef, Ein, Loeser, Looijenga, Mustata, and Veys; see for example \cite{Bat,Cr,DL1,DL2,EM,Loo,M,Ve}. More recently, arc spaces have turned out to have applications to the theory of vertex algebras, which in many cases can be viewed as quantizations of arc spaces \cite{A1,AKS,AM,LSSII,S1,S2,S3}.

 \subsection{Standard monomials for arc spaces} Let 
 \begin{equation} \label{def:ringRR} \R=\R_{p,q}=\mathbb Z[x^{(k)}_{ij}|\ 1\leq i\leq p, \ 1\leq j\leq q,\ k\geq 0],\end{equation} which has a derivation $\partial$ characterized by $\partial x_{ij}^{(k)}=(k+1)x_{ij}^{(k+1)}$. It can be regarded as the ring of polynomial functions with integer coefficients on the arc space of $p\times q$ matrices; in particular, $K[J_{\infty}(M_{p,q})] \cong \R \otimes_{\mathbb{Z}} K$.
 
Let $\R[h]$ be the ideal of $\R$ generated by all $h$-minors $B$ of the form (\ref{eqn:minor}) and their normalized derivatives $\frac 1 {n!}\partial^n B$.  Let 
\begin{equation} \label{def:ringRR_h} \R_h=\R\slash\R[h+1].\end{equation} Let $\cJ_r$ be the set of $h$-minors of the form (\ref{eqn:minor}) with $h\leq r$ and their normalized derivatives. Note that $R$ and $R_h$ are naturally subrings of $\R$ and $\R_h$, respectively. In Section \ref{sect:standardmonomial}, we will define a notion of standard monomial on $\R_h$ that extends the above notion on $R_h$, and we will prove the following result.

\begin{thm}\label{thm:standard} $\R_h$ has a $\mathbb Z$-basis given by the standard monomials of $\cJ_h$.
\end{thm}
Let $J_\infty(D_h)$ be the arc space of the determinantal variety  $D_h$. Then the affine coordinate ring $K[J_\infty(D_h)]$ is $\R_{h-1}\otimes_{\mathbb{Z}} K$, so we immediately have
\begin{cor}\label{cor:standarddeterm} $K[J_\infty(D_h)]$ has a $K$-basis given by the standard monomials of $\cJ_{h-1}$. \end{cor} 
  
When $K = \mathbb{C}$, the arc space $J_\infty(D_h)$, as well as the finite jet schemes $J_n(D_h)$, were also studied by Docampo in \cite{D}. He gave an explicit description of the decomposition of  $J_\infty(D_h)$ and  $J_n(D_h)$ as a union of orbits for the action of $J_{\infty}(GL_p(\mathbb{C}) \times GL_q(\mathbb{C}))$ and $J_{n}(GL_p(\mathbb{C}) \times GL_q(\mathbb{C}))$, respectively.

\subsection{Application in invariant theory}
Given an algebraic group $G$ over $K$, $J_{\infty}(G)$ is again an algebraic group. If $V$ is a finite-dimensional $G$-module, there is an induced action of $J_{\infty}(G)$ on $J_{\infty}(V)$, and the invariant ring $K[J_{\infty}(V)]^{J_{\infty}(G)}$ was studied in our earlier paper \cite{LSSI} with Schwarz in the case $K = \mathbb{C}$. 
The quotient morphism $V\to V/\!\!/G$ induces a morphism $J_\infty(V)\to J_\infty(V/\!\!/G)$, so we have a morphism \begin{equation} \label{equ:invariantmap} J_\infty(V)/\!\!/J_\infty(G)\to J_\infty(V/\!\!/G).\end{equation}
In particular, we have a ring homomorphism
\begin{equation} \label{equ:invariantringmap} K[J_{\infty}(V/\!\!/G)] \rightarrow K[J_{\infty}(V)]^{J_{\infty}(G)}.\end{equation} If $V/\!\!/G$ is smooth or a complete intersection and $K[V]$ has no nontrivial one-dimensional $G$-invariant subspaces, it was shown in \cite{LSSI} that \eqref{equ:invariantringmap} is an isomorphism, although in general it is neither injective nor surjective.

We specialize to the case $G=GL_h(K)$, $W=K^{\oplus h}$, and $V=W^{\oplus p}\bigoplus {W^*}^{\oplus q}$, as above. Then  $$K[J_{\infty}(V)] =K[a^{(k)}_{il}, b^{(k)}_{jl}|\ 1\leq i\leq p, \ 1\leq j\leq q,\ 1\leq l\leq h, \ k\in \Zplus],$$ which has an induced action of $J_\infty(GL_h(K))$ as above. We have the following theorem, which is the arc space analogue of Theorem \ref{thm:classicalmain}.
\begin{thm}\label{thm:main}
	Fix integers $h\geq 1$ and $p,q \geq 0$, and let $W = K^{\oplus h}$ and $V=W^{\oplus p}\bigoplus {W^*}^{\oplus q}$ be as above. Let $\bar\partial^k = \frac{1}{k!} \partial^k$ be the $k^{\text{th}}$ normalized derivative.
	\begin{enumerate}
		\item The ring of invariants $K[J_{\infty}(V)]^{J_\infty(GL_h(K))}$ is generated by 
		\begin{equation}\label{eqn:generators} \{X^{(k)}_{ij} = \bpartial^k\sum_l a^{(0)}_{il}b^{(0)}_{jl}|\ 1\leq i\leq p, \ 1\leq j\leq q, \ k \geq 0\}.\end{equation}
		\item The ideal of relations among the generators \eqref{eqn:generators} is generated by 
		\begin{equation}\label{eqn:minorX} \bpartial^k\left|
		\begin{array}{cccc}
		X^{(0)}_{u_1v_1} & X^{(0)}_{u_1v_2} & \cdots &X^{(0)}_{u_1v_{h+1}}\\
		X^{(0)}_{u_2v_1}& X^{(0)}_{u_2v_2} & \cdots & X^{(0)}_{u_2v_{h+1}} \\
		\vdots & \vdots & \vdots & \vdots \\
		X^{(0)}_{u_{h+1}v_1} & X^{(0)}_{u_{h+1}v_2} & \cdots & X^{(0)}_{u_{h+1}v_{h+1}} \\
		\end{array}
		\right|,
		\end{equation} for all $u_1, u_2,\dots, u_h$ and $v_1,v_2,\dots, v_h$ with $1\leq u_i < u_{i+1} \leq p$ and $1\leq v_i < v_{i+1} \leq q$, and all integers $k\geq 0$.
		\item  $K[J_{\infty}(V)]^{J_\infty(GL_h(K))}$ has a $K$-basis given by standard monomials of $\cJ_h$.
	\end{enumerate}
\end{thm}

\begin{cor}\label{cor:quotient} For all $h \geq 1$ and $p,q \geq 0$, the map $K[J_{\infty}(V/\!\!/GL_h(K))] \rightarrow K[J_{\infty}(V)]^{J_{\infty}(GL_h(K))}$ given by \eqref{equ:invariantringmap} is an isomorphism. In particular, we have
$$J_\infty(V)/\!\!/J_\infty(GL_h(K))\cong J_\infty(V/\!\!/GL_h(K)).$$
\end{cor}
Corollary \ref{cor:quotient} is a generalization of Theorem 4.6 of \cite{LSSI}, which deals with the following special cases for $K = \mathbb{C}$.
\begin{enumerate}
\item $p\leq h$ or $q \leq h$, so that $V/\!\!/GL_h(\mathbb C)$ is an affine space,
\item $p = h+1 = q$, so that $V/\!\!/GL_h(\mathbb C)$ is a hypersurface.
\end{enumerate}
In the second paper in this series \cite{LS1}, we will prove a similar theorem for the symplectic group $Sp_h(K)$ for $h$ an even integer: for $W = K^{\oplus h}$ and $V = W^{\oplus p}$, \eqref{equ:invariantringmap} is an isomorphism for all $h$ and $p$. In the third paper \cite{LS2}, we will study the case $G = SL_h(K)$, $W = K^{\oplus h}$ and $V=W^{\oplus p}\bigoplus {W^*}^{\oplus q}$. This case is more subtle since \eqref{equ:invariantringmap} is always surjective, but fails to be injective if $\text{max}(p,q) -2 > h$. We will completely determine its kernel, which coincides with the nilradical of $K[J_{\infty}(V/\!\!/G)]$ when $\text{char} \ K = 0$. Unfortunately we are unable to prove similar results for the orthogonal and special orthogonal groups using these methods.

Our results on the invariant theory of arc spaces have significant applications to vertex algebras which we will develop in separate papers \cite{LS3,LS4}. These include the structure of cosets of affine vertex algebras inside free field algebras, classical freeness of the affine vertex algebras $L_k(\mathfrak{sp}_{2n})$ for all positive integers $n$ and $k$, new level-rank dualities involving affine vertex superalgebras, and the complete description of the vertex algebra of global sections of the chiral de Rham complex of an arbitrary compact Ricci-flat K\"ahler manifold. 

\subsection{Acknowledgment:} B. Song would like to thank Mao Sheng for discussions and suggestions on algebraic geometry on this subject. A. Linshaw is supported by Simons Foundation Grant \#635650 and NSF Grant DMS-2001484. B. Song is supported  by NSFC No. 11771416.

\section{Standard monomials} \label{sect:standardmonomial}

Fix integers $p,q \geq 1$, and recall the ring 
$$\R=\R^{p,q}=\mathbb Z[x^{(k)}_{ij}|\ 1\leq i\leq p, \ 1\leq j\leq q,\ k\geq 0],$$ with derivation $\partial$ defined on generators by $\partial x_{ij}^{(k)}=(k+1)x_{ij}^{(k+1)}$. As above, this is an integral version of the coordinate ring of the arc space of the space of $p\times q$ matrices, that is, $K[J_{\infty}(M_{p,q})] = \R \otimes_{\mathbb{Z}} K$, for any field $K$.

\subsection{Normalized derivatives} 
For $\l \geq 0$, we define the $l^{\text{th}}$ normalized derivative $\bpartial^l=\frac 1 {l!}\partial^l$ on $\cR$, which satisfies
$$\bpartial^l x_{ij}^{(k)}=C_{k+l}^l  x_{ij}^{(k+l)}\in \R.$$
Here for $k,n \in \mathbb{Z}_{\geq 0}$, 
$$C_n^k=\left\{ \begin{array}{cc}
\frac{n!}{k!(n-k)!}, &0\leq k\leq n,\\
0, &\text{otherwise.}
\end{array}\right.
$$

The following propositions are easy to verify.
\begin{prop} For any $a,b\in\R$,
 $$\bpartial^l(ab)=\sum_{i=0}^l\bpartial^ia\, \bpartial^{l-i}b,$$
and $\bpartial^l a\in \R$.
\end{prop}
\begin{prop}\label{prop:detexpression}
	For a minor $B$ of the form \eqref{eqn:minor},
	\begin{equation}\bpartial^n B=\sum_{\substack{n_1+\cdots+n_h=n\\n_i\in \mathbb Z_{\geq 0}}}\sum_{\sigma}\sign(\sigma)x^{(n_1)}_{u_1v_{\sigma(1)}} x^{(n_2)}_{u_2v_{\sigma(2)}}\cdots x^{(n_h)}_{u_hv_{\sigma(h)}}.
	\end{equation}
\end{prop}

\subsection{Generators}  Recall that the minor $B$ in \eqref{eqn:minor} can be represented by the pair of ordered $h$-tuples $(u_h,\dots,u_2,u_1|v_1,v_2,\dots,v_h)$, where $1\leq u_i< u_{i+1}\leq p$ and $1\leq v_i< v_{i+1}\leq q$. Similarly, let
\begin{equation}\label{eqn:sequenceJ}
J=\bpartial^n(u_h,\dots,u_2,u_1|v_1,v_2,\dots,v_h)
\end{equation} represent $\bpartial^nB\in\R$, the $n^{\text{th}}$ normalized derivative of the minor $B$. For convenience, we shall call such expressions {\it $\bpartial$-lists} throughout this paper. We call $wt(J)=n$ the weight of $J$ and call $sz(J)=h$ the size of $J$. Let $\cJ$ be the set of these $\bpartial$-lists,
and
$$\cJ_h=\{J\in \cJ| sz(J)\leq h\}$$ 
be the set of elements of $\cJ$ with size less than or equal to $h$.
Let $\cE$ be the set of pairs of ordered $h$-tuples of ordered pairs of the form
\begin{equation}\label{eqn:element}
E=((u_h,k_h),\dots,(u_2,k_2),(u_1,k_1)|(v_1,l_1),(v_2,l_2),\dots,(v_h,l_h))
\end{equation}
with $1\leq u_i\leq p$, $1\leq v_i\leq q$, $u_i\neq u_j$ if $i\neq j$, $v_i\neq v_j$ if $i\neq j$, and $k_i,l_j\in \Zplus$. For each $E$, there are unique permutations $\sigma, \sigma'$ of $\{1,2,\dots, h\}$ such that $u_{\sigma(i)}< u_{\sigma(i+1)}$ and $v_{\sigma'(i)}<v_{\sigma'(i+1)}$.
Let
$$
||E||= \bpartial^n (u_{\sigma(h)},\dots,u_{\sigma(2)},u_{\sigma(1)}|v_{\sigma'(1)},v_{\sigma'(2)},\dots,v_{\sigma'(h)})\in\cJ.
$$
Here $n=\sum k_i+\sum l_i$ and $\sigma, \sigma'$ are the above permutations. Let 
$$wt(E)=wt(||E||),\quad\quad sz(E)=sz(||E||).$$
Let $$\cE_h=\{E\in\cE|\,\,sz(E)\leq h\}.$$
 For $J\in \cJ$, let
\begin{equation} \label{eq:defEJ} \cE(J)=\{E\in\cE|\ ||E||=J\}.\end{equation}
$\cJ$ is a set of generators of $\R$ and we can use the elements in $\cE(J)$ to represent $J$.

\subsection{Ordering} For any set $\cS$, let $\cM(\cS)$ be the set of ordered products of elements of $\cS$.
If $\cS$ is an ordered set, we order $\cM(\cS)$ lexicographically, that is 
$$
S_1S_2\cdots S_m\prec S'_1S'_2\cdots S'_n \quad\text{ if  } S_i=S'_i, i< i_0, \text{ with }S_{i_0}\prec S'_{i_0} \text { or } i_0=m+1,n>m.
$$ 
We order $\cM(\mathbb Z)$, the set of ordered product of integers, lexicographically.

There is an ordering on the set $\cJ$:
$$
\bpartial^k(u_h,\dots,u_2,u_1|v_1,v_2,\dots,v_h)\prec \bpartial^{k'}(u'_{h'},\dots,u'_2,u'_1|v'_1,v'_2,\dots,v'_{h'})
$$
if
\begin{itemize}
	\item
	$h'< h$ ;
	\item or $h'=h$ and $k< k'$;
	\item or  $h'=h$, $k= k'$  and $u_h\cdots u_1v_h\cdots v_1\prec u'_h\cdots   u'_1v'_h\cdots v'_1$. Here we order the words of natural numbers lexicographically.
\end{itemize}	
	
We order the pairs $(u,h)\in \Zplus\times \Zplus$ by
$$(u,h)\leq (u',h'), \text{ if } h<h' \text{ or } h=h' \text{ and }u\leq u'.$$
There is a partial ordering on the set $\cE$:
$$
((u_h,k_h),\dots,(u_1,k_1)|(v_1,l_1),\dots,(v_h,l_h))\leq ((u'_{h'},k'_{h'}),\dots,(u'_1,k'_1)|(v'_1,l'_1),\dots,(v'_{h'},l'_{h'}))
$$
if $h'\leq h$ and $(u_i,k_i)\leq (u'_i,k'_i)$, $(v_i,l_i)\leq (v'_i,l'_i)$, for $1\leq i\leq h'$.\\ 
Finally, there is an ordering on $\cE$:
$$ 
((u_h,k_h),\dots,(u_1,k_1)|(v_1,l_1),\dots,(v_h,l_h))\prec ((u'_{h'},k'_{h'}),\dots,(u'_1,k'_1)|(v'_1,l'_1),\dots,(v'_{h'},l'_{h'}))
$$
if
\begin{itemize}
	\item $h>h'$;
	\item or $h=h'$ and $\sum (k_i+l_i)<\sum (k'_i+l'_i)$;
	\item or  $h=h'$, $\sum (k_i+l_i)=\sum (k'_i+l'_i)$ and
	$$
	(u_h,k_h)\cdots (u_1,k_1)(v_h,l_h)\cdots (v_1,l_1)\prec (u'_{h'},k'_{h'})\cdots(u'_1,k'_1)(v'_{h'},l'_{h'})\cdots(v'_{1},l'_{1}).
	$$
	Here we order the words of $\Zplus\times \Zplus$ lexicographically.
\end{itemize}
\begin{lemma}\label{lemma:compare2}
	If $E\leq E'$, then $||E||\prec ||E'||$.
\end{lemma}
\begin{proof}
	If $sz(E')<sz(E)$ or $sz(E)=sz(E')$ and $wt(E)<wt(E')$, then $||E||\prec ||E'||$. If $sz(E)=sz(E')$ and $wt(E)=wt(E')$, we must have $k_i=k_i'$ and $l_j=l_j'$. So $u_i\leq u'_i$ and $v_j\leq v'_j$, we have $||E||\prec ||E'||$.
\end{proof}

\subsection{Relations}
Let $$\bpartial^k(u_h,\dots,u_1|v_1,\dots,v_h)=0$$ if there is $1\leq i<j\leq h$ such that  $u_i=u_j$  or $v_i=v_j$.  
For a $\bpartial$-list $J\in \cJ$ of the form \eqref{eqn:sequenceJ}, let 
$$\bpartial^k(u_{\sigma(h)},\dots,u_{\sigma(2)},u_{\sigma(1)}|v_{\sigma'(1)},v_{\sigma'(2)},\dots,v_{\sigma'(h)})=\sign(\sigma)\sign(\sigma')J.$$
Here $\sigma, \sigma'$ are permutations of $\{1,2,\dots,h\}$, and $\sign(\sigma)$, $\sign(\sigma')$ are the signs of these permutations.
 We have the following relations, which we will prove later in Section \ref{sec:proof1.2}.
\begin{lemma}\label{lem:fullrelation0}
   	For $i_1,i_2,j_1,j_2,h,h',k_0,m\in \mathbb Z_{\geq 0}$ with $h\geq h'$, $i_1,j_1\leq h$, $i_2,j_2\leq h'$ and $k_0\leq m$, let $l_0=i_1+i_2+j_1+j_2-2h-1$.  Given any integers  $a_{k}$, $k_0\leq k\leq k_0+l_0$, there are integers $a_k$, $0\leq k<k_0$ or $k_0+l_0<k\leq m$, such that 
	\begin{equation}\label{eqn: relation4'}
	\sum_{k=0}^ma_k\sum_{\sigma,\sigma'} \frac{1}{i_1!i_2!j_1!j_2!}\sign(\sigma) \sign(\sigma')
	\end{equation}
\begin{equation*}
	  \left(
	\begin{array}{ccc}
	\bpartial^{m-k} (u_h,\dots,u_{i_1+1},\sigma(u_{i_1}),\dots,\sigma(u_1)&|&\sigma'(v_1),\dots, \sigma'(v_{j_1}),v_{j_1+1},\dots,v_h)\\
	\bpartial^{k} (u'_{h'},\dots,u'_{i_2+1},\sigma(u'_{i_2}),\dots,\sigma(u'_1)&|&\sigma'(v'_1),\dots, \sigma'(v'_{j_2}),v'_{j_2+1},\dots,v'_{h'})
	\end{array}
	\right)\\
	\end{equation*}
is in $\R[h+1]$.
	Here the second summation is over all pairs of permutations $\sigma$ of $u_{i_1},\dots,u_1,$ $u'_{i_2},\dots,u'_1$ and permutations $\sigma'$  of $v_{i_1},\dots,v_1,v'_{i_2},\dots,v'_1$, and $\sign(\sigma)$ and $\sign(\sigma')$ are the signs of the permutations.
\end{lemma}

For simplicity, we write Equation (\ref{eqn: relation4'}) in the following way,
	\begin{equation}\label{eqn: relation4}
\sum \epsilon a_{k} \left(
\begin{array}{ccc}
\bpartial^{m-k} (u_h,\dots,u_{i_1+1},\underline{u_{i_1},\dots,u_1}&|&\underline{v_1,\dots, v_{j_1}},v_{j_1+1},\dots,v_h)\\
\bpartial^{k} (u'_{h'},\dots,u'_{i_2+1},\underline{u'_{i_2},\dots,u'_1}&|&\underline{v'_1,\dots, v'_{j_2}},v'_{j_2+1},\dots,v'_{h'})
\end{array}
\right)
\in \R[h+1].
\end{equation}

\begin{remark}
Since the second summation in Equation (\ref{eqn: relation4'}) is over all permutations, each monomial in the equation will appear $i_1!i_2!j_1!j_2!$ times, so the coefficient of each monomial will be $\pm a_k$.\end{remark}

\subsection{Standard monomials} Now we give a definition of the standard monomials of $\cJ$.
\begin{defn}\label{def:standard}
An ordered product $E_{1}E_{2}\cdots E_{m}$ of elements of $\cE$ is said to be standard if 
\begin{enumerate} 
\item $E_a\leq E_{a+1}$, $1\leq a<m$, 
\item  $E_{1}$ is the largest in $\cE(||E_{1}||)$ under the order $\prec$, where  $\cE(||E_{1}||)$ is defined by \eqref{eq:defEJ},
\item $E_{a+1}$ is the largest in $\cE(||E_{a+1}||)$ such that $E_a\leq E_{a+1}$.
\end{enumerate}
 An ordered product $J_1J_2\cdots J_m$ of elements of $\cJ$ is said to be standard if there is a standard ordered product $E_{1}E_{2}\cdots E_{m}$ such that $E_i\in \cE(J_i)$.
\end{defn}	
Let $\cS\cM(\cJ)\subset \cM(\cJ)$ be the set of standard monomials of $\cJ$; \\
let $\cS\cM(\cE)\subset \cM(\cE)$ be the set of standard monomials of $\cE$; \\
let $\cS\cM(\cJ_h)=\cM(\cJ_h)\cap\cS\cM(\cJ)$ be the set of standard monomials of $\cJ_h$; \\
let $\cS\cM(\cE_h)=\cM(\cE_h)\cap\cS\cM(\cE)$ be the set of standard monomials of $\cE_h$.

By Definition \ref{def:standard}, if $J_1J_2\cdots J_m$ is a standard monomial, the standard monomial $E_1\cdots E_m\in \cS\cM(\cE)$ corresponding to $J_1\cdots J_m$ is unique and $E_1$ has the form 
$$((u_h,wt(E_1)),(u_{h-1},0),\dots,(u_1,0)|(v_1,0),\dots,(v_h,0))\in \cE$$
 with $u_i<u_{i+1}$ and $v_i<v_{i+1}$. So the map
 $$\pi_h :\cS\cM(\cE_h) \to \cS\cM(\cJ_h  ),\quad E_1E_2\cdots E_m\mapsto ||E_1|| ||E_2|| \cdots ||E_m||$$
is a bijection.

We order $\cM(\cJ)$, the set of ordered products of elements of $\cJ$, lexicographically. The following lemma will be proved later in Section \ref{section:proof2.5}.
\begin{lemma}\label{lem:base0}
	If $J_1\cdots J_b\in\cM(\cJ)$ is not standard, $J$ can be written as a linear combination of elements of $\cM(\cJ)$ preceding $J_1\cdots J_{b-1}$ with integer coefficients.
\end{lemma}

Recall that $\R[h]$ denotes the ideal generated by $J\in \cJ$ with $sz(J)=h$, and $\R_h=\R\slash\R[h+1]$, as in \eqref{def:ringRR_h}. If $h\geq min\{p,q\}$, then $\cJ_h=\cJ$ and $\R_h=\R$.
By the above lemma, we immediately have
\begin{lemma}\label{lem:base1}
	Any element of $\R_h$ can be written as a linear combination of standard monomials of $\cJ_h$ with integer coefficients.
\end{lemma}
\begin{proof} We only need to show that any element of $\R$ can be written as a linear combination of standard monomials of $\cJ$ with integer coefficients. Recall that $\cJ$ generates $\R$. If the lemma is not true,
there must be a smallest element $J\in \cM(\cJ)$, which cannot be written as a linear combination of elements of $\cS\cM(\cJ)$ with integer coefficients. So $J$ is not standard. By Lemma \ref{lem:base0},
$J=\sum_{\alpha}\pm  J_{\alpha}$ with $J_{\alpha}\in \cM(\cJ)$ and $J_{\alpha}\prec J$. Since $J_{\alpha}$ can be written as a linear combination of elements of $\cS\cM(\cJ)$ with integer coefficients, $J$ can also be written as such a linear combination, which is a contradiction.
\end{proof}

\section{A canonical basis}
\subsection{A ring homomorphism}
Let  
$$
\cS_h=\{a^{(k)}_{il},b^{(k)}_{jl}|\ 1\leq i\leq p, \ 1\leq j\leq q, \ 1\leq l\leq h, \ k\in \Zplus \},
$$
and let \begin{equation} \label{def:ringB} \B=\mathbb Z [\cS_h],\end{equation} the polynomial ring generated by $\cS_h$. For later use, we mention that for a field $K$, if $W = K^{\oplus h}$, and $V = W^{\oplus p} \oplus W^{*\oplus q}$, the affine coordinate ring $K[J_{\infty}(V)]$ is obtained from $\B$ by base change, i.e., $K[J_{\infty}(V)] = \B \otimes_{\mathbb{Z}} K$.

Let $\partial$ be the derivation on $\B$ given by
$\partial a^{(k)}_{ij}=(k+1)a^{(k+1)}_{ij}$, $\partial b^{(k)}_{ij}=(k+1)b^{(k+1)}_{ij}$. 
We have a homomorphism of rings
$$\tilde Q_h: \R\to \B,\quad\quad  x^{(k)}_{ij}\mapsto \bpartial^k\sum_{l=1}^h a^{(0)}_{il}b^{(0)}_{jl}.$$
For any $J\in \cJ$ with $sz(J)>h$, we have $\tilde Q_h(J)=0$, so $\tilde Q_h$ induces a ring homomorphism 
\begin{equation} \label{def:Qh} Q_h:\R_h\to \B. \end{equation}

\subsection{Double tableaux}
Let $\tilde \cS_h$=$\cS_h\cup\{*\}$.
We define an ordering on the set $\tilde\cS_h$: 
for $X^{(k)}_{ij}, Y^{(k')}_{i'j'}\in \cS_h$,\\
$X^{(k)}_{ij}<*$ 
and
$X^{(k)}_{ij}\geq Y^{(k')}_{i'j'}$
 if
\begin{itemize}
	\item $X=a$, $Y=b$; 
	\item or $X=Y$, $k>k'$; 
	\item or $X=Y$, $k=k'$, $i>i'$; 
	\item or $X=Y$, $k=k'$, $i=i'$, $j\geq j'$.
\end{itemize}

We use double tableaux to represent the monomials of $\B$. Let $\cT$ be the set of the following double tableaux:
\begin{equation}\label{eqn:table}
\left|\begin{array}{ccc}
y_{1,h},\cdots,y_{1,2}, y_{1,1} & | & z_{1,1},z_{1,2},\cdots, z_{1,h} \\
\vdots & \vdots & \vdots \\
y_{m,h},\cdots,y_{m,2}, y_{m,1} & | & z_{m,1},z_{m,2},\cdots, z_{m,h} \\
\end{array}\right|.
\end{equation}
Here $y_{s,l}$ are some $a^{(k)}_{il}$ or  $*$ and $z_{s,l}$ are some $b^{(k)}_{jl}$ or $*$; every row of the tableau has elements in $\cS_h$; and
 $$y_{s,j}\leq y_{s+1,j}, \quad z_{s,j}\leq z_{s+1,j}.$$
 We use the tableau (\ref{eqn:table}) to represent a monomial in $\B$, which is the product of  ${a^{(k)}_{ij}}'s$ and ${b^{(k)}_{ij}}'s$ in the tableau.
It is easy to see that the representation is a one-to-one correspondence between $\cT$ and the set of monomials of $\B$.
We associate to the tableau (\ref{eqn:table}) the word:
$$ 
 y_{1,h}\cdots y_{1,1} z_{1,h}\cdots z_{1,1} y_{2,h}\cdots  y_{2,1} z_{2,h}\cdots z_{2,1}\cdots z_{m,h}\cdots z_{m,1}
 $$
and order these words lexicographically.
For a polynomial $f\in \B$, let $Ld(f)$ be its leading monomial in $f$ under the order we defined on $\cT$.

For $E_i=((u^i_{h_1},k^i_{h_1}),\dots,(u^i_2,k^i_2),(u^i_1,k^i_1)|(v^i_1,l^i_1),(v^i_2,l^i_2),\dots,(v^i_{h_1},l^i_{h_1}))\in \cE$, $1\leq i\leq m$,
we use a double tableau to represent $E_1\cdots E_m\in\cS\cM(\cE)$,
\begin{equation}\label{eqn:tableE} \left(
\begin{array}{ccc}
(u^1_{h_1},k^1_{h_1}),\cdots,(u^1_2,k^1_2),(u^1_1,k^1_1)&|&(v^1_1,l^1_1),(v^1_2,l^1_2),\cdots,(v^1_{h_1},l^1_{h_1})\\
(u^2_{h_2},k^2_{h_2}),\cdots,(u^2_2,k^2_2),(u^2_1,k^2_1)&|&(v^2_1,l^2_1),(v^2_2,l^2_2),\cdots,(v^2_{h_2},l^2_{h_2})\\
\vdots&\vdots&\vdots \\
(u^m_{h_m},k^m_{h_m}),\cdots,(u^m_2,k^m_2),(u^m_1,k^m_1)&|&(v^m_1,l^m_1),(v^m_2,l^m_2),\cdots,(v^m_{h_m},l^m_{h_m})\\
\end{array}
\right).
\end{equation}
Let $T:\cS\cM(\cJ_h) \to \cT$ with
$$ T(E_1\cdots E_m)=\left(
\begin{array}{ccc}
*,\cdots, *, a_{u^1_{h_1}h_1}^{(k^1_{h_1})},\cdots,a_{u^1_{1}1}^{(k^1_{1})}&|&b_{v^1_{1}1}^{(l^1_{1})},\cdots,b_{v^1_{h_1}h_1}^{(l^1_{h_1})},*,\cdots,*\\
*,\cdots, *, a_{u^2_{h_2}h_2}^{(k^2_{h_2})},\cdots,a_{u^2_{1}1}^{(k^2_{1})}&|&b_{v^2_{1}1}^{(l^2_{1})},\cdots,b_{v^2_{h_2}h_2}^{(l^2_{h_2})},*,\cdots,*\\
\vdots&\vdots&\vdots \\
*,\cdots, *, a_{u^m_{h_m}h_m}^{(k^m_{h_m})},\cdots,a_{u^m_{1}1}^{(k^m_{1})}&|&b_{v^m_{1}1}^{(l^m_{1})},\cdots,b_{v^m_{h_m}h_m}^{(l^m_{h_m})},*,\cdots,*
\end{array}
\right).
$$
Obviously, $T$ is an injective map and $T(E_1)\prec T(E_2)$ if $E_1\prec E_2$.
\begin{lemma}\label{lem:leadingterm}
	Let $J_1\cdots J_m\in \cS\cM(\cJ_h)$ and $E_1\cdots E_m\in \cS\cM(\cE_h)$ be its associated standard monomial. Assume the double tableau representing $E_1\cdots E_m$ is (\ref{eqn:tableE}). Then
	the leading monomial of $Q_h(J_1\cdots J_m)$ is represented by the double tableau $T(E_1E_2\cdots E_m)$.  
	Thus $$Ld\circ Q_h=T\circ \pi_h^{-1}: \cS\cM(\cJ_h)\to \cT$$  is injective. Moreover, the coefficient of the leading monomial of $Q_h(J_1\cdots J_m)$ is $\pm 1$.
\end{lemma}
\begin{proof}
	Let $W_m$ be the monomial corresponding to the tableau $T(E_1\cdots E_m)$. Let $$
	M_m=a_{u^m_{h_m}h_m}^{(k^m_{h_m})}\cdots a_{u^m_{1}1}^{(k^m_{1})}b_{v^m_{1}1}^{(l^m_{1})}\cdots b_{v^m_{h_m}h_m}^{(l^m_{h_m})}
	$$
	 be the monomial corresponding to the double tableau $T(E_m)$.
	Then $W_m=W_{m-1}M_m$.
	We prove the lemma by induction on $m$. If $m=1$, the lemma is obviously true. Assume the lemma is true for $J_1\cdots J_{m-1}$, then its leading monoimal $Ld(Q_h(J_1\cdots J_{m-1}))=W_{m-1}$, the monomial corresponding to $T(E_1\cdots E_{m-1})$, and the coefficient of $W_{m-1}$ in $Q_h(J_1\cdots J_{m-1})$ is $\pm 1$.
	$$
	Q_h(J_m)	=\sum \pm a_{u^m_{1} s_1}^{(k_1)}a_{u^m_{2} s_2}^{(k_2)}\cdots a_{u^m_{h_m}s_{h_m}}^{(k_{h_m})}b_{v^m_{1} t_1}^{(l_1)}b_{v^m_{2}t_2}^{(l_2)}\cdots b_{v^m_{h_m} t_{h_m}}^{(l_{h_m})}. 
	$$
	The summation is over all $l_i, k_i\geq 0$ with $\sum (l_i+k_i)=wt(E_m)$, all $s_i$ with $1\leq s_1, s_2,\cdots,s_{h_m}\leq h$ and
	they are different from each other, and all $t_1,\cdots, t_{h_m}$, which are permutations of  $s_1,s_2,\cdots,s_{h_m}$.
	$M_m$ is one of the monomials in $Q_h(J_m)$ with coefficient $\pm 1$.
	 All of the monomials in the polynomial $Q_h(J_1\cdots J_{m-1})$ except $W_{n-1}$ are less than
	 $W_{n-1}$, so any monomial in $Q_h(J_1\cdots J_{m-1})$ except $W_{n-1}$ times any monomial in $Q_h(J_m)$, is less than $W_{m-1}$. Since $W_{m-1}\prec W_m$, the coefficient of $W_m$ in $Q_h(J_1\cdots J_{m})$ is 
	 not zero.  Now $$W_{m-1}\prec W_m\prec Ld(Q_h(J_1\cdots J_{m})).$$
The leading monomial $Ld(Q_h(J_1\cdots J_{m}))$ must have the form
	$$W=W_{m-1}a_{u^m_{1} s_1}^{(k_1)}a_{u^m_{2} s_2}^{(k_2)}\cdots a_{u^m_{h_m}s_{h_m}}^{(k_{h_m})}b_{v^m_{1} t_1}^{(l_1)}b_{v^m_{2}t_2}^{(l_2)}\cdots b_{v^m_{h_m} t_{h_m}}^{(l_{h_m})}.$$
	If some $s_i$ or $t_i$ greater than $h_{m-1}$, then $W\prec W_{n-1}$.
	 If there is some $h_{m-1}\geq s_i>h_m$, there is $1\leq j\leq h_m$, with $j\notin\{s_1,\dots, s_{h_m}\}$, if we replace $s_i$ by $j$ in $W$, we get a larger monomial in $Q_h(J_1\cdots J_{m})$. So we can assume $s_1,\dots, s_{h_m}$ is a permutation of $1,2,\dots, h_m$. 
	We must have $a_{u^m_{i}s_i}^{(k_{i})}\geq a_{u^{m-1}_{s_i}s_i}^{(k^{m-1}_{s_i})}$ and $b_{v^m_{i}t_i}^{(k_{i})}\geq b_{v^{m-1}_{t_i}t_i}^{(k^{m-1}_{t_i})}$, otherwise $W\prec W_{m-1}$. These kind of monomials in 
$Q_h(J_m)$ are in one-to-one correspondence with $E'_m\in\cE(J_m)$ such that $E_{m-1}\leq  E'_m$. Finally, $E_{m}$ is the largest in $\cE(J_m)$ with $E_{m-1}\leq E_m$ since $E$ is standard, so $W_m$ is the leading term of $Q_h(J_1\cdots J_m))$. The coefficient of $W_m$ in $Q_h(J_1\cdots J_m)$
is $\pm 1$ since the coefficients of $W_{m-1}$ in $Q_h(J_1\cdots J_{m-1})$ and $M_m$ in $Q_h(J_m)$ are $\pm 1$. 
\end{proof}

\begin{proof}[Proof of Theorem \ref{thm:standard}.]
 By Lemma \ref{lem:leadingterm}, $Ld(Q_h(\cS\cM(\cJ_h)))$ are linearly independent, so $\cS\cM(\cJ_h)$ is a linearly independent set. By Lemma \ref{lem:base1}, $\cS\cM(\cJ_h)$ generates $\R_h$. So $\cS\cM(\cJ_h)$ is a $\mathbb Z$-basis of $\R_h$.
\end{proof}

\begin{thm}\label{thm:standmonomial1} 
 $Q_h:\R_h\to \B$ is injective. So we may identify $\R_h$ with the image $\text{Im}(Q_h)$, which is the subring of $\cB$ generated by $\bpartial^k\sum_{i=1}^r a^{(0)}_{il}b^{(0)}_{jl}$. In particular, $Q_h(\cS\cM(\cJ_h))$ is a $\mathbb Z$-basis of $\text{Im}(Q_h)$.
\end{thm}

\begin{proof}
By Lemma \ref{lem:leadingterm}, $Ld(Q_h(\cS\cM(\cJ_h)))$ are linearly independent. Since $\cS\cM(\cJ_h)$ is a $\mathbb Z$-basis of $\R_h$, $Q_h:\R_h\to \B$ is injective. \end{proof}

Since $Q_h$ is injective and $\B$ is an integral domain, we obtain
\begin{cor} $\R_h$ is an integral domain.
\end{cor}

\section{Application}
In this section, we give the main application of the standard monomial basis we have constructed, which is the arc space analogue of Theorem \ref{thm:classicalmain}.

\subsection{Arc spaces} Suppose that $X$ is a scheme of finite type over $K$. Its arc space (cf.\cite{EM}) $J_\infty(X)$ is determined by its functor of points. For every $K$-algebra $A$, we have a bijection
$$\Hom(\Spec A, J_\infty(X))\cong\Hom(\Spec A[[t]], X).$$
If $i: X\to Y$ is a morphism of schemes, we get a morphism of schemes $i_{\infty}:J_{\infty}(X)\to J_{\infty}(Y)$.
If $i$ is a closed immersion, then $i_{\infty}$ is also a closed immersion.

If $X=\Spec K[x_1,\dots,x_n]$, then $J_{\infty}(X)=\Spec K[x^{(k)}_i|1\leq i\leq n, k\in \Zplus]$. The identification is made as follows: for a $K$-algebra $A$, a morphism $\phi: K[x_1,\dots, x_n]\to A[[t]]$ determined by $\phi(x_i)=\sum_{k=0}^\infty a_i^{(k)}t^k$ corresponds to a morphism
$K[x_i^{(k)}]\to A$ determined by $x_i^{(k)}\to a_i^{(k)}$. Note that $K[x_1,\dots,x_n]$ can naturally be identified with the subalgebra $K[x^{(0)}_1,\dots,x^{(0)}_n]$ of $K[x_i^{(k)}]$, and from now on we use $x_i^{(0)}$ instead of $x_i$.

The polynomial ring $K[x^{(k)}_i]$ has a derivation $\partial$ defined on generators by 
\begin{equation} \label{def:partial} \partial x_i^{(k)}=(k+1)x_i^{(k+1)}.\end{equation}
It is more convenient to work with the normalized $k$-derivation $\frac 1{k!}\partial^k$, but this is a priori not well-defined on $K[x_i^{(k)}]$ if $\text{char}\ K$ is positive. However, $\partial$ is well-defined on $\mathbb Z[x^{(k)}_{i}]$ and $\bpartial^k = \frac 1{k!}\partial^k$ maps $\mathbb Z[x_i^{(k)}]$ to itself, so for any $K$, there is an induced $K$-linear map 
\begin{equation} \label{def:bpartial} \bpartial^k:K[x_i^{(k)}]\to K[x_i^{(k)}],\end{equation} obtained by tensoring with $K$.

\begin{prop}
	If $X$ is the affine space $\Spec K[x_1^{(0)},\dots, x_n^{(0)}]/(f_1,\dots, f_r)$, then $J_\infty(X)$ is an affine space  $$\Spec K[x_1^{(0)},\dots,x^{(0)}_n,\dots,x^{(k)}_i,\dots]/(f_1,\dots,f_r,\bpartial f_1,\dots, \bpartial^kf_j,\dots).$$
\end{prop}
\begin{proof}
	Let $\bpartial^k: A[[t]] \to A[[t]]$ be a morphism of $A$-modules with $\bpartial^k t^n=C_n^k t^{n-k}$. Then for any $a(t), b(t)\in A[[t]]$, we have $$\bpartial^n(a(t)b(t))=\sum_{k=0}^n\bpartial^k a(t) \bpartial^{n-k}b(t),$$ and the coefficient of $t^k$ in $a(t)$ is $\bpartial^k a(t)|_{t=0}$ . Any morphism $$\phi: K[x^{(0)}_1,\dots, x^{(0)}_n]\to A[[t]]$$ determined by $\phi(x^{(0)}_i)=\sum_{k=0}^\infty a_i^{(k)}t^k$ induces a morphism  
	$$\tilde \phi: K[x^{(k)}_i]\to A[[t]], \quad \text{given by }x^{(k)}_i \mapsto \bpartial^k\phi(x_i^{(0)}).$$
	Then $\tilde \phi  \bpartial^k=\bpartial^k  \tilde \phi$ and $\tilde \phi(x^{(k)}_i)|_{t=0}=a^{(k)}_i$. 
	
	For every $f\in K[x^{(0)}_i]$, $$\bpartial^k \phi(f)|_{t=0}=\tilde \phi(\bpartial^k f))|_{t=0}=(\bpartial^k f) (\tilde\phi (x^{(0)}_{1}),\dots, \tilde\phi(x^{(k)}_{n}))|_{t=0}=(\bpartial^kf)(a_1^{(0)},\dots, a_n^{(k)}),$$ we have $$\phi(f)=\sum_{k=0}^\infty \bpartial^kf(a_1^{(0)},\dots, a_n^{(k)}) \,t^k.$$
	It follows that $\phi$ induces a morphism $K[x^{(0)}_i]/(f_1,\dots,  f_r)\to A[[t]]$ if and only if 
$$\bpartial^kf_i(a_1^{(0)},\dots, a_n^{(k)})=0,\ \text{for all} \ i = 1,\dots, r,\ k\geq 0.$$
\end{proof}

If $Y$ is the affine scheme $\text{Spec} \ K[y^{(0)}_1,\dots,y^{(0)}_m]/(g_1,\dots,g_s)$, a morphism $P:X\to Y$ induces a ring homomorphism $P^*:K[Y]\to K[X]$.
Then the induced morphism of arc spaces $P_\infty:J_\infty(X)\to J_\infty(Y)$ is given by $P^*_\infty(y_i^{(k)})=\bpartial^k P^*(y_i^{(0)})$; in particular, $P^*_\infty$ commutes with $\bpartial^k$ for all $k\geq 0$.

\subsection{Arc space of the determinantal variety}
Recall that the space $M_{p,q}$ of $p\times q$ matrices over $K$ has affine coordinate ring
$$K[M_{p,q}] =K[x^{(0)}_{ij}|\ 1\leq i\leq p, \ 1\leq j\leq q],$$ which is just $R \otimes_{\mathbb{Z}} K$, where $R$ is given by \eqref{def:ringR}. 
The determinantal variety $D_h$ is the subvariety of $M_{p,q}$ determined by the ideal $K[M_{p,q}][h]$ generated by all $h$-minors, so $K[D_h] = K[M_{p,q}] / K[M_{p,q}][h] = R_{h-1}\otimes_{\mathbb{Z}} K$, where $R_{h-1}$ is given by \eqref{def:ringR_h}. Similarly, recall that $$K[J_{\infty}(M_{p,q})] = K[x^{(k)}_{ij}|\ 1\leq i\leq p, \ 1\leq j\leq q, \ k\in \mathbb{Z}_{\geq 0}] = \R \otimes_{\mathbb{Z}} K,$$ where $\R$ is given by \eqref{def:ringRR}. Then
$$K[J_{\infty}(D_h)] =K[J_{\infty}(M_{p,q})] / K[J_{\infty}(M_{p,q})][h],$$ where $K[J_{\infty}(M_{p,q})][h]$ is the ideal generated by the elements $\bpartial^n J$, where $J$ is an $h$-minor. Note that $K[J_{\infty}(D_h)] = \R_{h-1}\otimes_{\mathbb{Z}} K$, where $\R_{h-1}$ is given by \eqref{def:ringRR_h}.

\begin{proof}[Proof of Corollary \ref{cor:standarddeterm}]
	By Theorem \ref{thm:standard}, $\cS\cM(\cJ_{h-1})$ is a $\mathbb Z$-basis of $\R_{h-1}$. So it is a $K$-basis of $K[J_{\infty}(D_h)]$.
\end{proof}

\subsection{Invariant theory for $J_{\infty}(GL_h(K))$}
Let $G=GL_h(K)$ be the general linear group of degree $h$ over $K$.
The group structure $G\times G\to G$
induces the group structure on its arc space
$$J_\infty(G)\times J_\infty(G)\to J_\infty(G),$$ 
so $J_\infty(G)$ is an algebraic group. Recall the $G$-modules $W=K^{\oplus h}$ and $V=W^{\oplus p}\bigoplus {W^*}^{\oplus q}$. Recall that $V$ has affine coordinate ring 
$$K[V] = K[a^{(0)}_{il},b^{(0)}_{jl}|\ 0\leq i\leq p,\ 1\leq j\leq p,\ 1\leq l\leq h].$$
The action $G\times V\to V$ induces the action of $J_\infty(G)$ on
$J_\infty (V)$,
$$J_\infty(G)\times J_\infty(V)\to J_\infty(V).$$
This induces an action of $J_{\infty}(G)$ on the affine coordinate ring
$$K[J_{\infty} (V)] = K[a^{(k)}_{il},b^{(k)}_{jl}|\ 0\leq i\leq p,\ 1\leq j\leq p,\ 1\leq l\leq h, \ k\in \mathbb{Z}_{\geq 0}],$$ which is identified with is $\B \otimes_{\mathbb{Z}} K$ where $\B$ is given by \eqref{def:ringB}.

Recall the map $Q_h: \R_h \to \B$ given by \eqref{def:Qh}. It extends naturally to a map
\begin{equation} \label{def:QhK} Q^K_h: K[J_{\infty}(D_{h+1})] \to K[J_{\infty} (V)],\end{equation} where $K[J_{\infty}(D_{h+1})]$ and $K[J_{\infty} (V)]$ are identified with $ \R_{h}\otimes_{\mathbb{Z}} K$ and $\B \otimes_{\mathbb{Z}} K$, respectively, and $Q_h^K=Q_h\otimes Id$.

\begin{thm}\label{thm:injectiveGL} $Q^K_h$ is injective, so we may identify $K[J_{\infty}(D_{h+1})]$ with the subring $\text{Im}(Q^K_h)$ of $K[J_{\infty} (V)]$. In particular, $K[J_{\infty}(D_{h+1})]$ is integral.
\end{thm}
\begin{proof}
	By Lemma \ref{lem:leadingterm}, $Ld(Q_h(\cS\cM(\cJ_h)))$ are linearly independent. By Corollary \ref{cor:standarddeterm}, $\cS\cM(\cJ_h)$ is a $K$-basis of $\R_h$,  so $Q^K_h$ is injective. Since $K[J_{\infty}(V)]$ is integral, so is $K[J_{\infty}(D_{h+1})]$.
\end{proof}
\begin{remark} In general, if $\text{char}\  K=0$, the arc space of an integral scheme is irreducible \cite{Kol}, but it may not be reduced. The determinantal varieties are examples whose arc spaces are integral.\end{remark}

 If $p,q \geq h$, 
let $\Delta=Q^K_{h}((h, \dots, 1|1,\dots, h))$. Let $K[J_{\infty}(V)]_{\Delta}$ and  $\text{Im}(Q^K_h)_{\Delta}$ be the localizations of $K[J_{\infty}(V)]$ and $\text{Im}(Q^K_h)$ at $\Delta$, respectively.

\begin{lemma}\label{lemma:equalGLinvariant}If $p,q\geq h$,
	$$K[J_{\infty}(V)]_\Delta^{J_\infty({GL}_h(K))} = \text{Im}(Q^K_h)_\Delta.$$ 
\end{lemma}
\begin{proof} Let $K[V]_\Delta$ be the localization of $K[V]$ at $\Delta$ and
	$V_\Delta=\Spec K[V]_\Delta$. Let $H$ be the subvariety of $V_\Delta$ given by the ideal generated by $a_{il}-\delta_i^l$ with $1\leq i,l\leq h$. 
	The composition of the imbedding $\iota: H\hookrightarrow V_\Delta$ and the affine quotient $q: V_\Delta \to V_\Delta/\!\!/G=(V/\!\!/G)_\Delta$ gives the isomorphism $q\circ \iota:H \to (V/\!\!/G)_\Delta$. So as morphisms of their arc spaces, 
	$$q_\infty\circ \iota_\infty :J_\infty(H) \to J_\infty((V/\!\!/G)_\Delta)=J_\infty(V/\!\!/G)_\Delta.$$ 
	
	The map $q_\infty$ induces a morphism
	$\bar q_\infty: J_\infty(V_\Delta)/\!\!/J_{\infty}(G) \to J_\infty(V/\!\!/G)_\Delta$.
	The action of $G$ on $V$ gives a $G$-equivariant isomorphism
	$$G\times H\to V_\Delta.$$
	So we have a $J_{\infty}(G)$-equivariant isomorphism
	$$J_{\infty}(G) \times J_\infty(H)\to J_\infty(V_\Delta)=J_\infty(V)_\Delta.$$
	and an isomorphism of their affine quotients
	$$i: J_\infty(H)=J_{\infty}(G) \times J_\infty(H)/\!\!/G\cong J_\infty(V)/\!\!/J_{\infty}(G).$$
	$\bar q_\infty \circ i=q_\infty\circ \iota_\infty$ is an isomorphism,
	so $\bar q_\infty$ is an isomorphism since $i$ is an isomorphism, which is equivalent to the lemma. \end{proof}
	
\begin{thm}\label{thm:JGLinvariant}
	$K[J_{\infty}(V)]^{J_\infty({GL}_h(K))} = \text{Im}(Q^K_h)$. 
\end{thm}
\begin{proof} If $p,q\geq h$, we regard $K[J_{\infty}(V)]$ and $\text{Im}(Q^K_h)_{\Delta}$ as subrings of $K[J_{\infty}(V)]_{\Delta}$.
	By Lemma \ref{lemma:equalGLinvariant}, we have  $$K[J_{\infty}(V)]^{J_{\infty}(G)}= K[J_{\infty}(V)] \cap \text{Im}(Q^K_h)_{\Delta}.$$ Now for any $f \in K[J_{\infty}(V)] \cap \text{Im}(Q^K_h)_{\Delta}$, 
	$f=\frac{g}{\Delta^n}$ with $\Delta^nf=g\in \text{Im}(Q^K_h)$. The leading monomial of $g$ is $$Ld(g)=(a^{(0)}_{11}\cdots a^{(0)}_{hh} b^{(0)}_{11}\cdots b^{(0)}_{hh})^n Ld(f)$$
	with coefficient $C_0\neq 0$.
	Since $g\in \text{Im}(Q^K_h)$, there is a standard monomial $J\in\cS\cM(\cJ_h)$, with $Ld(Q_h(J))=Ld(g)$. Since $J$ has the factor $(h, \dots, 1|1,\dots, h)^n$, $Q^K_h(J)$ has the factor $\Delta^n$.
	Thus $f-C_0\frac{Q^K_h(J)}{{\Delta}^n}\in K[J_{\infty}(V)]  \cap \text{Im}(Q^K_h)_{\Delta}$ with a lower leading monomial and $\frac{Q^K_h(J)}{{\Delta}^n}\in \text{Im}(Q^K_h)$. By induction on the leading monomial of $f$, $f\in \text{Im}(Q^K_h)$. So $$K[J_{\infty}(V)] \cap \text{Im}(Q^K_h)_{\Delta}=\text{Im}(Q^K_h),$$ and  $K[J_{\infty}(V)]^{J_{\infty}(G)} =\text{Im}(Q^K_h)$.
	
	More generally, let $V' = W^{\oplus p + h} \bigoplus (W^*)^{\oplus q+h}$, where $W = K^{\oplus h}$ as before. Its arc space has affine coordinate ring 
	$$K[J_{\infty}(V')] = K[a^{(k)}_{il},b^{(k)}_{jl}|\ 1\leq i\leq p+h,\ 1\leq j\leq q+h,\ k\in \Zplus],$$ which contains $K[J_{\infty}(V)]$ as a subalgebra, and has an action of $J_{\infty}(G)$. By the above argument, $K[J_{\infty}(V')]^{J_{\infty}(G)}$ is generated by $X^{(k)}_{ij}=\bpartial^k\sum_l a_{il}b_{jl}$. Let $\cI$ be the ideal of $K[J_{\infty}(V')]$ generated by $a^{(k)}_{il}, b^{(k)}_{jl}$ with $i>p, j>q$. Then 
	$$K[J_{\infty}(V')]= K[J_{\infty}(V)] \oplus {\cI}.$$ Note that $K[J_{\infty}(V)]$ and $\cI$ are $J_{\infty}(G)$-invariant subspaces of $K[J_{\infty}(V')]$, and
	$$K[J_{\infty}(V')]^{J_{\infty}(G)} =K[J_{\infty}(V)]^{J_{\infty}(G)} \oplus {\cI}^{J_{\infty}(G)}.$$
	If $i>p$ or $j>q$, $X^{(k)}_{ij}\in {\cI}^{G_\infty}$. So 
	$$K[J_{\infty}(V)]^{J_{\infty}(G)} \cong K[J_{\infty}(V')]^{J_{\infty}(G)} \slash {\cI}^{J_{\infty}(G)}$$ is generated by $X^{(k)}_{ij}$, $1\leq i\leq p$, $1\leq j\leq q$. It follows that $K[J_{\infty}(V)]^{J_{\infty}(G)}=\text{Im}(Q^K_h)$, as claimed.
\end{proof}

\begin{proof}[Proof of Theorem \ref{thm:main}]
	By Theorem \ref{thm:JGLinvariant} and Theorem \ref{thm:injectiveGL}, $K[J_{\infty}(V)]^{J_\infty(GL_h(K))}=\text{Im}(Q^K_h)\cong K[J_{\infty}(D_{h+1})]$.
	\end{proof}

\begin{proof}[Proof of Corollary \ref{cor:quotient}]
This is immediate from Theorem \ref{thm:main} because $V/\!\!/GL_h(K)$ is isomorphic to the determinantal variety $D_{h+1}$. \end{proof}

\section{Some properties of standard monomials }
By the definition of standard monomials, if $E_1E_2\cdots E_n\in \cS\cM(\cE)$, then $E_{i+1}$ is the largest element in $||\cE(E_{i+1})||$ such that $E_i\leq E_{i+1}$. In this section, we study the properties of $||\cE(E_{i+1})||$ and $E_{i+1}$ that need to be satisfied to make $E_1E_2\cdots E_n$ a standard monomial.

Let 
$$E=((u_h,k_h),\dots,(u_1,k_1)|(v_1,l_1),\dots,(v_h,l_h))\in \cE,$$
$$ 
J'=\bpartial^{n'}(u'_{h'},\dots,u'_1|v'_1,\dots,v'_{h'})\in \cJ.$$
 
\subsection{ $L(E,J')$ and $R(E,J')$}

 For $h'\leq h$,
let $\sigma_L$ and $\sigma_R$ be the permutations of $\{1,2,\dots,h'\}$  such that $u_{\sigma_L(i)}<u_{\sigma_L(i+1)}$ and $v_{\sigma_R(i)}<v_{\sigma_R(i+1)}$.
Let $L(E,J')$ and $R(E,J')$ be the smallest non-negative integers $i_0$ and $j_0$ such that 
$u'_{i}\geq u_{\sigma_L(i-i_0)}$, $i_0< i \leq h'$ and  $v'_{j}\geq v_{\sigma_L(j-j_0)}$, $j_0<j\leq h'$, respectively.
Let
\begin{equation} \label{def:Ehprime} E(h')=((u_{h'},k_{h'}),\dots,(u_1,k_1)|(v_1,l_1),\dots,(v_{h'},l_{h'})).\end{equation}
Then $L(E,J')=L(E(h'),J')$ and $R(E,J')=R(E(h'),J')$.

The following lemma is obvious.
\begin{lemma}\label{lemma:replace}
For $J''=\bpartial^k(u''_{h'},\dots,u''_1|v''_{1},\dots,v''_{h'})\in\cJ$, if there are at least $s$ elements in $\{u''_{h'},\dots, u''_1\} $ from the set  $\{u'_{h'},\dots, u'_1\} $, then 
$L(E,J'')\geq L(E,J')-h'+s$;  
   if there are at least $s$ elements in $\{v''_{h'},\dots, v''_1\} $ from the set  $\{v'_{h'},\dots, v'_1\} $, then 
$R(E,J'')\geq R(E,J')-h'+s$.

\end{lemma}
\subsection{A criterion for $J'$ to be greater than $E$}
We say
$J'$ is greater than $E$ if there is an element $E'\in\cE(J')$ with $E\leq E'$. Then $J'$  is greater than $E$  if and only if $J'$  is greater than $E(h')$. The following lemma is a criterion for $J'$ to be greater than $E$.
\begin{lemma} \label{lemma:critgreat}
		$J'$ is greater than $E$ if and only if 
	$wt(J')-wt(E(h'))\geq L(E,J')+R(E,J')$.
\end{lemma}
\begin{proof} 
	Let
	$i_0=L(E,J')$ and $j_0=R(E,J')$. Let $\sigma$ and $\sigma'$ be permutations of $\{1,2,\dots,h'\}$  such that $u_{\sigma_L(i)}<u_{\sigma_L(i+1)}$ and  $v_{\sigma_R(i)}<v_{\sigma_R(i+1)}$.

	If $wt(J')-wt(E(h'))\geq L(E,J')+R(E,J')$,	
	let $$\tilde u'_{\sigma(i)}= \left\{  \begin{array}{cc}
	u'_{i+i_0}, & \sigma(i)+i_0\leq h' \\
	u'_{i+i_0-h_b},&   i+i_0> h'
	\end{array}\right.,
	\quad  k'_{\sigma(i)}= \left\{  \begin{array}{cc}
	k_{\sigma(i)}, & i+i_0\leq h', i\neq h'\\
	k_{\sigma(i)}+1,&   i+i_0> h', i\neq h'
	\end{array}\right. ,
	$$
	$$\tilde v'_{\sigma'(j)}= \left\{ \begin{array}{cc}
	v'_{j+j_0}, & j+j_0\leq h' \\
	v'_{j+j_0-h_b},&   j+j_0> h'
	\end{array}\right.,
	\quad \quad l'_{\sigma(j)}= \left\{ \begin{array}{cc}
	l_{\sigma'(j)}, & j+j_0\leq h' \\
	l_{\sigma'(j)}+1,&   j+j_0> h'
	\end{array} \right.,
	$$
	$$k'_{\sigma(h')}=wt(J')-\sum_{i=1}^{h'-1}k'_i-\sum_{j=1}^{h'}l'_i.$$
	Then
	$$k'_{\sigma(h')}=wt(J')-wt(E(h'))-i_0-j_0+k_{\sigma(h')}+1-\delta_{i_0}^{0}\geq k_{\sigma(h')}+1-\delta_{i_0}^{0}.$$
	$$(\tilde u'_{\sigma(i)},k'_{\sigma(i)})\geq (u_{\sigma(i)},k_{\sigma(i)}),\quad (\tilde v'_{\sigma'(j)},l'_{\sigma'(j)})\geq (v_{\sigma'(j)},l_{\sigma'(i)}). $$
	So $$\tilde E'=((\tilde u'_{h'},k'_{h'}),\dots,(\tilde u'_{2},k'_2),(\tilde u'_{1},k'_1)|(\tilde v'_{1},l'_1),(\tilde v'_{2},l'_2),\dots,(\tilde v'_{h'},l'_{h'}))$$
	is an element in $\cE(J')$ with $\tilde E'\geq E$.
	
	On the other hand, suppose that $\tilde E'\in \cE(J')$ with $\tilde E'\geq E$.
	Assume $$\tilde E' =((\tilde u'_{h'},k'_{h'}),\dots,(\tilde u'_{2},k'_2),(\tilde u'_{1},k'_1)|(\tilde v'_{1},l'_1),(\tilde v'_{2},l'_2),\dots,(\tilde v'_{h'},l'_{h'})).$$
	We have
	$(\tilde u'_i,k'_i)\geq (u_i,k_i)$ i.e.
	$k'_i>k_i$ or $k'_i=k_i, \tilde u'_i\geq u_i$
	and
	$(\tilde v'_i,l'_i)\geq (v_i,v_i)$ for $1\leq i\leq h'$ i.e.
	$l'_i>l_i$ or $l'_i=l_i, \tilde v'_i\geq v_i$.
	So
	$$\sum_{i=1}^{h'}(k_i'-k_i)+\sharp\{\tilde u'_i\geq u_i, i|1\leq i\leq h'\}\geq h', $$
	$$\sum_{i=1}^{h'}(l_i'-l_i)+\sharp\{\tilde v'_i\geq v_i, i|1\leq i\leq h'\}\geq h'.$$
	Let $i'_0=h'-\sharp\{\tilde u'_i\geq u_i, i|1\leq i\leq h'\}$ and $j'_0=h'-\sharp\{\tilde v'_i\geq v_i, i|1\leq i\leq h'\}$. Then 
	$$i'_0+j'_0\leq \sum_{i=1}^{h'}(k_i'-k_i)+\sum_{i=1}^{h'}(l_i'-l_i)= wt(J')-wt(E(h')).$$
	Here $\tilde u_1',\dots,\tilde u'_{h'}$ is a permutation of $u'_1,\dots,u'_{h'}$ and  $\tilde v_1',\dots,\tilde v'_{h'}$ is a permutation of $v'_1,\dots,v'_{h'}$.
	By the definition of $i'_0$ and $j'_0$, it is easy to see that  $u'_{i}\geq u_{\sigma(i-i'_0)}$, $i'_0< i \leq h'$ and  $v'_{j}\geq v_{\sigma'(j-j'_0)}$, $j'_0<j\leq h'$. So $i'_0\geq L(E,J')$ and $j'_0\leq R(E,J')$. Thus
	$$wt(J')-wt(E(h'))\geq i_0'+j_0'\geq L(E,J')+R(E,J').$$
	
\end{proof}

\begin{cor} \label{cor:critgreat1}
	$J'$ is greater than $E$ if and only if $||E(h')||J'$ is standard.
\end{cor}
\begin{proof}
	By Lemma \ref{lemma:critgreat}, $J'$ is greater than $E$ if and only if $wt(J')-wt(E(h'))\geq L(E,J')+R(E,J')$ and $||E(h')||J'$ is standard if and only if $wt(J')-wt(E(h'))\geq L(E(h'),J')+R(E(h'),J')=L(E,J')+R(E,J')$.
\end{proof}

\subsection{The property \lq\lq largest"}	
Let
\begin{equation*}
\cW_s(E,J')=\{J=\bpartial^k(u'_{i_s},\dots , u'_{i_1}|v'_{j_1},\dots , v'_{j_s})| \ 1\leq i_l,j_l\leq h',\ J \text{ is greater than } E\}.
\end{equation*}

\begin{lemma}\label{cor:compare}
If $E'$ is the largest element in $\cE(J')$ such that $E\leq E'$, then for $s<h'$,
 $||E'(s)||$ is the smallest element in $\cW_s(E,J')$.
	\end{lemma}
\begin{proof}
	Assume $$E'=(( u'_{h'},k_{h'}),\dots,(u'_{2},k_2),( u'_{1}, k_1)|(v'_1,l'_1),(v'_2,l'_2),\dots,(v'_h,l'_h)).$$
	For $s<h'$, let $J_s$ be the smallest element in $\cW_s(E,J')$. Let
	$$E_s=(( u'_{i_s},\tilde k_{s}),\dots,(u'_{i_2},\tilde k_2),( u'_{i_1},\tilde k_1)|( v'_{j_1},\tilde l_{1}),(v'_{j_2},\tilde l_2),\dots,( v'_{j_s},\tilde l_s))$$
	 be the largest element in $\cE(J_s)$ such that $E(s)\leq E _s$. 
	 
	 Assume $l$ is the largest number such that $(u'_j, k'_j)=(u'_{i_j},\tilde k_j)$ for $j<l\leq s+1$. 
	\begin{enumerate}
		\item 
	 If $l\leq s$, then $i_l\geq l$ and $(u'_{i_l},\tilde k_l)\neq (u'_l,k'_l)$. If $i_l=l$, by the maximality of $E'$ and the minimality of $J_s$, we must have $(u'_{i_l},\tilde k_l)= (u'_l,k'_l)$. This is a contradiction, so $i_l>l$.
	 
	 If $(u'_{i_l},\tilde k_l)<(u'_l,k'_l)$, then
	$(u'_{l}, k'_l+k'_{i_l}-\tilde k_{l})> (u'_{i_l}, k'_{i_l})$.
	Let $E''$ be the element in $\cE(J')$ obtained by replacing $( u'_{l},k'_l)$ and $(u'_{i_l},k'_{i_l})$ in $E'$
	by $(u'_{i_l},\tilde k_{l})$ and   $(u'_{l}, k'_l+k'_{i_l}-\tilde k_{l})$, respectively. We have $E'\prec E''$ and $E(h')\leq E''$.
	But $E'\neq E''$ is the largest element in $\cE(||E'||)$ such that $E\leq E'$, which is a contradiction.

	Assume $(u'_{i_l},\tilde k_l)>(u'_l,k'_l)$. If $l\notin\{i_1,\dots,i_s\}$, replacing $(u'_{i_l},\tilde k_l)$ in $E_s$ by $(u'_l,k'_l)$, we get $E'_s$ with $E(s)\leq E'_s$ and $||E'_s||\prec J_s$. This is impossible since $J_s\neq ||E'_s||$
	is the smallest element in $\cW_s(E,E') $. If $l=i_j \in\{i_1,\dots,i_s\}$,
	$(u'_{i_l}, \tilde k_l+\tilde k_{j}- k'_{l})> (u'_{i_j}, \tilde k_{j})$.
	Let $E'_s$ be the element in $\cE(J_s)$ obtained by replacing $(u'_{i_l},\tilde k_l)$ and $(u'_{i_j}, \tilde k_{j})$ in $E'_s$
	by $(u'_l,k'_l)$ and   $(u'_{i_l}, \tilde k_l+\tilde k_{j}- k'_{l})$, respectively. We have $E_s\prec E'_s$ and $E\leq E'_s$.
	But $E'_s\neq E_s$ is the largest element in $\cE(J_s)$ such that $E\leq E'_s$, a
	contradiction.
	\item
	 If $l=s+1$, then the left part of $E_s$ is equal to the left part of $E'(s)$.  Assume $m$ is the largest number such that $(v'_i, l'_i)=(v'_{j_i},\tilde k_i)$ for $i<m\leq s+1$. By the same argument of (1), we can show that $m=s+1$. 
	 	\end{enumerate}
 	So $E_s=E'(s)$ and $||E'(s)||$ is the smallest element in $\cW_s(E,J')$.
\end{proof}
\begin{cor}\label{cor:lrnumber1}
If $E'$ is the largest element in $\cE(||E'||)$ such that $E\leq E'$, then for $s<h'$,
$$L(E,||E'(s)||)+R(E,||E'(s)||)=wt(E'(s))-wt(E(s)).$$	
\end{cor}
\begin{proof}
Since $E\leq E'$, $E\leq E'(s)$. 
By Lemma \ref{cor:compare}, $||E'(s)||$ is the smallest element in $\cW_s(E,J')$. By Lemma \ref{lemma:critgreat}, $L(E,||E'(s)||)+R(E,||E'(s)||)=wt(E'(s))-wt(E(s)).$
\end{proof}
\begin{cor}\label{cor:lrnumber2}
If $E'$ is the largest element in $\cE(||E'||)$ such that $E\leq E'$, then for $s<h'$ and any	$J\in \cW_s(E,E')$,
	$$L(E,||E'(s)||)\leq L(E,J), \quad R(E,||E'(s)||)\leq R(E,J).$$
\end{cor}
\begin{proof}
 Assume $$J=\bpartial^k(u_s,\dots, u_1|v_1 ,\dots, v_s), \quad ||E'(s)||=\bpartial^l(u'_s,\dots , u'_1|v'_1, \dots , v'_s).$$
  If $m=L(E,J)-L(E,||E'(s)||)<0$, let $J''= \bpartial^{l+m}(u_s, \dots , u_1|v'_1, \dots, v'_s)$. By Lemma \ref{lemma:critgreat},  $J''\in W_s(E,J')$. By Lemma \ref{cor:compare}, $||E'(s)||$ is the smallest element in $\cW_s(E,||E'||)$. But $wt(||E'(s)||)>wt(J'')$, a contradiction. Similarly, we can show  $R(E,||E'(s)||)\leq R(E,J)$. \end{proof}

\begin{lemma}\label{lem:lrnumber}
	Let $E_i=((u^i_{h_i},k^i_{h_i}),\dots,(u^i_1,k^i_1)|(v^i_1,l^i_1),\dots,(v^i_{h_i},l^i_{h_i}))$, $i=a,b$. Suppose that $E_b\leq E_a$ and that $E_a$ is the largest element in $\cE(||E_a||)$ such that $E_b\leq E_a$. Let $1\leq h<h_a$ and $\sigma_i$, $\sigma'_i$ be permutations of $\{1,\dots, h\}$, such that $u^i_{\sigma_i(1)}<u^i_{\sigma_i(2)}<\cdots<u^i_{\sigma_i(h)}$ and $v^i_{\sigma'_i(1)}<v^i_{\sigma'_i(2)}<\cdots<v^i_{\sigma'_i(h)}$. Let $v_1',\dots,v'_{h_a}$ be a permutation of $v^a_1,\dots,v^a_{h_a}$ such that $v_1'<v_2'<\cdots<v'_{h_a}$.  Let $u_1',\dots,u'_{h_a}$ be a permutation of $u^a_1,\dots,u^a_{h_a}$ such that $u_1'<u_2'<\cdots<u'_{h_2}$.
	\begin{enumerate}
		\item \label{item:1}
		Assume $u'_{i_2}=u^a_{\sigma(i_1)}$ with $i_2>i_1$, then for any
		$$K=\bpartial^k(u_h,\dots, u_{s+1},u'_{t_s},\dots, u'_{t_1}|v_{1},\dots,v_h),$$
		with $t_1<t_2<\cdots <t_s<i_2$, $L(E_b,K)>L(E_b,||E_a(h)||)+s-i_1$. 
		\item \label{item:2}
		Assume $v'_{j_2}=v^a_{\sigma(j_1)}$ with $j_2>j_1$, then for any
		$$K=\bpartial^k(u_h,\dots,u_1|v'_{t_1},\dots,v'_{t_s},v_{s+1},\dots,v_h),$$
		with  $t_1<t_2<\cdots <t_s<j_2$, $R(E_b,K)>R(E_b,||E_a(h)||)+s-j_1$. 
	\end{enumerate}
\end{lemma}

\begin{proof}
	We prove (\ref{item:2}). The proof for (\ref{item:1}) is similar.\\
	 Let $n=wt(E_a(h))$.\\
	Let $j_0=R(E_b,K)$, then
	$v'_{t_j}\geq v^{b}_{\sigma_{b}(j-j_0)}$, for $s\geq j>j_0$.\\
	 Let $j'_0= R(E_b, ||E_a(h)||)$, then
	$v^a_{\sigma_a(j)}\geq v^{b}_{\sigma_{b}(j-j'_0)}$, for $h\geq j>j'_0$.\\
	We have $j_1> j'_0$.
	Otherwise, $j_1\leq j'_0$. Replacing $v^a_{\sigma_a(j_1)}$ in $||E^a(h)||$ by some $v^a_j<v^a_{\sigma_a(j_1)}$ with $j>h$, (such $v^a_j$ exists since $j_2>j_1$), we get  $$J=\bpartial^{n}(u_{\sigma_{a}(h)}^{a},\dots, u^{a}_{\sigma_{a}(1)}|v_{\sigma'_{a}(1)}^{a},\dots, v^a_j,\dots ,v^{a}_{\sigma'_{a}(h)})
	$$ with
	$R(E_b,J)\leq j'_0.$
	By Lemma \ref{lemma:critgreat}, $J$ is greater than $E_b$.
	This is impossible by Lemma \ref{cor:compare} since 
    $J\prec ||E_a(h)||$ and 
   $E_a$ is the largest element in $\cE(||E_a||)$ such that $E_b\leq E_a$.

	If 
	$s\geq j_1$, let $$J'= \bpartial^{n}(u_{\sigma_{a}(h)}^{a}, \dots , u^{a}_{\sigma_{a}(1)}|v'_{t_{s-j_1+1}}, \dots , v'_{t_s} v_{\sigma_{a}(j_1+1)}^{a}, \dots , v_{\sigma_{a}(h)}^{a} ).$$
 If $R(E_b,K)\leq R(E_b,||E_a||)+s-j_1$,
$$v'_{t_j}\geq v^{b}_{\sigma_{b}(j-j_0)}\geq v^{b}_{\sigma_{b}(j-j'_0-s+j_1)}, \,\,\text{for }j\leq s.$$	
	 We have
	$R(E_b,J')\leq j'_0.$ By Lemma \ref{lemma:critgreat}, $J'$ is greater than $E_b$. By Lemma \ref{cor:compare}, this is impossible since 
	$J'\prec ||E_a(h)||$ and $E_a$ is the largest element in $\cE(||E_a||)$ such that $E_b\leq E_a$.
	So when $s\geq j_1$, $$R(E_b,K)\geq R(E_b,||E_a||)+s-j_1.$$

	 If $s<j_1$, let $t'_1<t'_2<\cdots<t'_{j_1}<j_2$ with $\{t_1,\dots, t_s\}\subset \{t'_1,\dots, t'_{j_1}\}$. Let $$K'=\bpartial^k(u_h,\dots,u_1|v'_{t'_1},\dots,v'_{t'_{j_1}},v_{j_1+1},\dots,v_h).$$
	By Lemma \ref{lemma:replace}, we have $$R(E_b,K)\geq E(E_b,K')+s-j_1
	\geq  R(E_b,||E_a||)+s-j_1.$$
	
	So for any $s>0$, we have $R(E_b,K)> R(E_b,||E_a||)+s-j_1$.
\end{proof}
The following lemmas are obvious.
\begin{lemma}\label{lemma:order1}If $J_1\prec J_2\prec \cdots \prec J_n$, and $\sigma$ is a permutation of $\{1,\cdots ,n\}$, then
	$$J_1J_2\cdots J_n\prec J_{\sigma(1)}\cdots J_{\sigma(n)}.$$
\end{lemma}
\begin{lemma} \label{lemma:order2}If $K_1K_2\cdots K_k \prec J_1\cdots J_l$, then
	$$K_1K_2\cdots K_{s-1}JK_s\cdots K_k \prec J_1\cdots J_{s-1}JJ_s\cdots J_l.  $$
\end{lemma}

\section{Proof of Lemma (\ref{lem:fullrelation0}) }\label{sec:proof1.2}
In this section we will prove  Lemma (\ref{lem:fullrelation0}). Since the relation in Lemma (\ref{lem:fullrelation0}) does not depend on $p$ and $q$ if $p, q\geq h+h'$, we can assume $p=q=s=h+h'$.
Let $S=\{1,2,\dots, s\}$. For a subset $N\subset S$, let $|N|$ be the number of elements of $N$. Let $\bar N=S\backslash N$.
For $l\in S$, let $\partial_l$ and $\tpartial_l$ be the differentials on $\R$ given by
$$
\partial_l  x_{ij}^{(k)}=\delta^l_i  (k+1)x_{ij}^{(k+1)}, \quad \tpartial_l  x_{ij}^{(k)}=\delta^l_j (k+1) x_{ij}^{(k+1)}.
$$
Let 
$$\partial_N=\sum_{l\in N} \partial_l,\quad\quad \tpartial_N=\sum_{l\in N} \tpartial_l.$$
So $\partial=\partial_S=\tpartial_S$.
Let
$$
\bpartial^l_i=\frac 1 {l!}\partial^l_i,\quad\quad \bpartial^l_N=\frac 1 {l!}\partial^l_N,\quad \quad\bar\tpartial^l_N=\frac 1 {l!}\tpartial^l_N.
$$
For $a,b\in\R$, we have $\bpartial^l_N(ab)=\sum_{i=0}^l\bpartial_N^ia\, \bpartial_N^{l-i}b$ and $\bpartial_N^l a$ and $ \bar\tpartial_N^l a\in \R$.

If $I=\{i_1,\dots,i_k\}\subset S$ and $ J=\{j_1,\dots,j_k\}\subset S$ with $i_a<i_{a+1}$ and $j_a<j_{a+1}$,
let $$\cA(I,J)=\left|
\begin{array}{cccc}
x_{i_1j_1}^{(0)} & x_{i_1j_2}^{(0)} & \cdots &x_{i_1j_k}^{(0)}\\
x_{i_2j_1}^{(0)}& x_{i_2j_2}^{(0)} & \cdots & x_{i_2j_k}^{(0)} \\
\vdots & \vdots & \vdots & \vdots \\
x_{i_kj_1}^{(0)} & x_{i_kj_2}^{(0)} & \cdots & x_{i_kj_k}^{(0)} \\
\end{array}
\right|
$$ be the determinant of the $k\times k$  matrix with entries $x_{ij}^{(0)}$ for $i\in I$ and $j\in K$.
For example, $$\cA(\{1,2\},\{1,3\})=x_{11}^{(0)}x_{23}^{(0)}-x_{13}^{(0)}x_{21}^{(0)}.$$
Let $\epsilon (I,J)=(-1)^{\sum_{i\in I} i+\sum_{j\in J}j}$.

\begin{lemma}\label{lem:relation1} For $I, K, L\subset S$ with $L\subset I$, $|K|=|I|=k$ and $|L|\leq n\leq |I|$, if $l<2n-|L|$,
	$$ \sum_{L\subset N\subset  I, |N|=n}\bpartial^{l}_{N}\cA(I,K)\in\R[k-n+1].$$
\end{lemma}
\begin{proof}We say $a\sim b$ if $a-b\in\R[k-n+1]$. 
	It is an equivalence relation on $\R$.
	$$\bpartial^{l}_{N}\cA(I,K)=\sum_{\substack{\sum_{i\in N}l_i=l\\ l_i\geq 0} }(\prod_{i\in N}  \bpartial_i^{l_i})\cA(I,K).$$
	We have the following properties.\\
	\textsl{Property 1:}
	if there is some $i_0\in N$ with $l_{i_0}=0$, then  $(\prod_{i\in N} \bpartial_i^{l_i})\cA(I,K)\in \R[k-n+1]$.
	\\
	Since
	$$
	(\prod_{i\in N}  \bpartial_i^{l_i})\cA(I,K)=\sum_{J\subset K,|J|=k-n+1} \pm\cA((I\backslash N)\cup\{i_0\},J)(\prod_{i\in N} \bpartial_i^{l_i})\cA(N\backslash\{i_0\}, K\backslash J).
	$$
	\textsl{Property 2:} if $L\subset M \subset I$ with $|M|=m<n$ and  $\sum_{i\in M}l_i=l-n+m$, then
	\begin{equation}\label{eqn:multsim}
	\sum_{M\subset N\subset I, |N|=n}(\prod_{j\in N\backslash M}  \partial_j) (\prod_{i\in M} \bpartial_i^{l_i})\cA(I,K)\in \R[k-n+1].
	\end{equation}
	Since on one hand by property 1,
	$$\bpartial^{n-m}_{\bar M}(\prod_{i\in M}  \bpartial_i^{l_i})\cA(I,K)
	\sim \sum_{M\subset N\subset I, |N|=n}(\prod_{j\in N\backslash M} \partial_j)(\prod_{i\in M} \bpartial_i^{l_i})\cA(I,K);$$
	and on the other hand,
	$$\bpartial^{n-m}_{\bar M}(\prod_{i\in M} \bpartial_i^{l_i})\cA(I,K)=\sum_{J\subset K, |J|=|M|} \pm(\prod_{i\in M} \bpartial_i^{l_i})\cA(M,J)\bpartial^{n-m}\cA( I\backslash M,K\backslash J)$$
	$$\in \R[k-m].$$
	Now
	$$ \sum_{L\subset N\subset  I, |N|=n} \bpartial^{l}_{N}\cA(I,K)=\sum_{L\subset N\subset  I, |N|=n} \sum_{\sum_{i\in N} l_i=l}(\prod_{i\in N} \bpartial_i^{l_i})\cA(I,K)$$
	(take out $\bpartial_j^{l_j}$ with $l_j=1$ and $j\notin L$.)
	$$=\sum_{L\subset N\subset  I, |N|=n} \sum_{L\subset M\subset N}\sum_{\substack{l_i\neq 1, i\in M\backslash L;\\ \sum_{i\in M} l_i=l-n+|M| }}(\prod_{j\in N\backslash M} \partial_j) (\prod_{i\in M} \bpartial_i^{l_i})\cA(I,K)$$
	(switch the order of the summation)
	$$=\sum_{L\subset M\subset  I, |M|\leq n}\sum_{\substack{l_i\neq 1, i\in M\backslash L;\\ \sum_{i\in M} l_i=l-n+|M| }} \sum_{ M\subset N\subset I, |N|=n}(\prod_{j\in N\backslash M} \partial_j)(\prod \bpartial_i^{l_i})\cA(I,K)$$
	(by property 2)
	$$\sim \sum_{L\subset M\subset  I, |M|=n}\sum_{\substack{l_i\neq 1, i\in M\backslash L;\\ \sum_{i\in M} l_i=l }} (\prod \bpartial_i^{l_i})\cA(I,K)$$
	(by property 1, since $l<2n-|L|$, there must some $l_i=0$)
	$$\sim 0$$
\end{proof}
\begin{lemma}\label{lem:relation2}For $T, J, K\subset S$ with $J\cap T=\emptyset$, if $0\leq a\leq l$, then
	$$\sum_{\substack{|I|=|K|\\J\subset I\subset\bar T}}\epsilon(I,K)\bpartial^a\cA(I,K)\bpartial^{l-a}_{T}\cA(\bar I, \bar K)\in \R[s-|J|-|T|-a].$$
\end{lemma}
\begin{proof}
	Let $\cB(T,J,K)$ be the determinant of the $s\times s$ matrix with entries $y_{ij}$, where
	$$y_{ij}=x_{ij}^{(0)}, \text{ if }(i,j)\notin (T\times K) \cup (J\times \bar K); \quad  y_{ij}=0,\text{ if } (i,j)\in (T\times K) \cup (J\times \bar K).$$
	On one hand,
	$$\bpartial^{l-a}_T\bar \tpartial^a_K \cB(T,J,K)=\sum_{\substack{|I|=|K|\\J\subset I\subset\bar T}}\epsilon(I,K)\bpartial^a\cA(I,K)\bpartial^{l-a}_{T}\cA(\bar I, \bar K).$$
	On the other hand
	$\bpartial^{l-a}_T\bar\tpartial^a_K \cB(T,J,K)$
	$$=\sum_{b=0}^a\sum_{I\subset \bar K, |I|=|T|}\sum_{M\subset K,|M|=|J|}\pm \bpartial^{l-a}\cA(T,I) \bpartial^b\cA(J,M)\bar \tpartial^{a-b}_K \cA(\bar T\cap \bar J,\bar I\cap\bar M).$$
	It is easy to see that $$\bar \tpartial^{a-b}_K\cA(\bar T\cap \bar J,\bar I\cap\bar M)\in \R[s-|J|-|T|-a].$$
	So
	$$\sum_{\substack{|I|=|K|\\J\subset I\subset\bar T}}\epsilon(I,K)\bpartial^a\cA(I,K)\bpartial^{l-a}_{T}\cA(\bar I, \bar K)\in \R[s-|J|-|T|-a].$$
	
\end{proof}
\begin{lemma}\label{lem:relation3} For $L, K\subset S$ with $|L|, |K|\leq s- n$, if $0\leq l\leq 2(s-n)-|L|-|K|-1$, then
	$$\sum_{N\subset \bar L,|N|=n} \sum_{J\subset \bar K, |J|=n} \epsilon (N,J)\cA(N,J) \bpartial^l\cA(\bar N,\bar J)\in \R[n+1].$$
\end{lemma}
\begin{proof}Let $|K|=k$.
	For $N\subset S$ with $|N|=n$,
	let $\cD(N,K)$ be the determinant of  the $s\times s$ matrix with entries
	$$
	y_{ij}=x_{ij}^{(0)}, \text { if }i\notin N \text{ or } j\notin K; \quad y_{ij}=0, \text { if } i\in N \text{ and  } j\in K.
	$$
	We have
	$$\cD(N,K)=\sum_{J\subset \bar K, |J|=n} \epsilon (N,J)\cA(N,J) \cA(\bar N,\bar J)$$
	and
	$$\cD(N,K)=\sum_{I\subset \bar N, |I|=k} \epsilon (I,K)
	\cA(I,K)\cA(\bar I,\bar K).$$
	By taking the summation of $\bpartial^l\cD(N,K)$ over $N\subset\bar L$, we have
	\begin{eqnarray}
	&&\sum_{N\subset \bar L,|N|=n} \sum_{J\subset \bar K, |J|=n}\epsilon (N,J) \cA(N,J) \bpartial^l \cA(\bar N,\bar J)\nonumber\\
	&=& \sum_{N\subset \bar L,|N|=n}\bpartial^l_{\bar N}\cD(N,K)\nonumber\\
	&=&\sum_{N\subset \bar L,|N|=n} \sum_{I\subset \bar N, |I|=k}\sum_{a=0}^l\epsilon (I,K)\bpartial^a\cA(I,K) \bpartial^{l-a}_{\bar N\cap\bar I}\cA(\bar I,\bar K)\nonumber\\
	& &\text{(switching the order of the summation)}\nonumber\\
	&=&\sum_{a=0}^l\sum_{I\subset S, |I|=k}\epsilon (I,K)\bpartial^a\cA(I,K)\sum_{N\subset \bar I\cap \bar L, |N|=n} \bpartial^{l-a}_{\bar N\cap\bar I}\cA(\bar I,\bar K).\nonumber
	\end{eqnarray}
	If $l-a<2(s-k-n)-|\bar I\cap L|$, let $J=\bar N\cap \bar I$. By Lemma \ref{lem:relation1},
	$$\sum_{N\subset \bar I\cap \bar L, |N|=n} \bpartial^{l-a}_{\bar N\cap \bar I}\cA(\bar I,\bar K)
	=\sum_{ L\cap \bar I\subset J \subset \bar I , |J|=s-k-n} \bpartial^{l-a}_{J}\cA(\bar I,\bar K)\in \R[n+1].$$
	If $l-a\geq 2(s-k-n)-|\bar I\cap L|$, 	let $J=I\cap L$, $T=\bar N\cap \bar I$.	By Lemma \ref{lem:relation2},
	$$\sum_{I\subset S, |I|=k}\epsilon (I,K)\bpartial^a\cA(I,K)\sum_{N\subset \bar I\cap \bar L, |N|=n} \bpartial^{l-a}_{\bar N\cap\bar I}\cA(\bar I,\bar K)$$
	$$
	=\sum_{\substack{J\subset L,T\subset \bar J\\|T|=s-n-k} }\sum_{\substack{T\subset \bar I\subset\bar J,\\ |I|=|K|}}\bpartial^a\cA(I,K)\bpartial^{l-a}_{T}\cA(\bar I,\bar K)\in\R[n+1].$$
	This completes the proof.
\end{proof}

For $L_0,L_1, K_0,K_1\subset S$ with $L_0\cap L_1=\emptyset=K_0\cap K_1$, $|L_0|+|L_1|\leq s- n$ and $|K_0|+|K_1|\leq s- n$,
let
$$\cF^k_l(L_0,L_1,K_0,K_1,n)=\sum_{\substack{|N|
		=n\\L_0\subset N\subset \bar L_1}} \sum_{\substack{|J|=n\\K_0\subset J\subset \bar K_1}} \epsilon (N,J)\bpartial^l\cA(N,J) \bpartial^{k-l}\cA(\bar N,\bar J).$$
We have
\begin{equation}\label{eqn:partialF}\bpartial^m \cF^k_l(L_0,L_1,K_0,K_1,n)=\sum_{a=0}^m C_{l+a}^lC_{k+m-l-a}^{k-l}\cF^{k+m}_{l+a}(L_0,L_1,K_0,K_1,n).\end{equation}

\begin{lemma}\label{lem:relation4}  If $l\leq 2(s-n)-|L_0|-|L_1|-|K_0|-|K_1|-1$, then
	$$\cF^l_0(L_0,L_1,K_0,K_1,n)\in \R[n+1].$$
\end{lemma}
\begin{proof}We can show this by induction on $|L_0|+|K_0|$.
	If $|L_0|+|K_0|=0$, $L_0=K_0=\emptyset$. By Lemma \ref{lem:relation3}, the lemma is true.

	Suppose the lemma is true for $|L_0|+|K_0|=m$.
	For $|L_0|+|K_0|=m+1$,
	assume $L_0\neq \emptyset$ (similarly for $K_0\neq \emptyset$). Let $i_0\in L_0$ and $L=L_0\backslash\{i_0\}$,
	$$\cF^l_0(L_0,L_1,K_0,K_1,n)=\cF^l_0(L,L_1,K_0,K_1,n)-\cF^l_0(L,L_1\cup\{i_0\},K_0,K_1,n)\in \R[n+1] $$
	by induction.
\end{proof}
Now we have the relations for the determinant of the matrix.
\begin{lemma}\label{lem:fullrelation}
	For $L_0,L_1, K_0,K_1\subset S$ with $L_0\cap L_1=\emptyset=K_0\cap K_1$, $|L_0|+|L_1|\leq s- n$ and $|K_0|+|K_1|\leq s- n$, let $l_0= 2(s-n)-|L_0|-|L_1|-|K_0|-|K_1|-1$.
	Given a fixed integer $0 \leq k_0\leq m-l_0$ and integers $a_{m,{k_0+l}}$, $0\leq l\leq l_0$, there are integers $a_{m,k}$, $0\leq k<k_0$ or $k_0+l_0<k\leq m$ such that
	$$\sum_{k=0}^ma_{m,k}\cF^m_k(L_0,L_1, K_0,K_1,n)\in\R[n+1].$$
	
\end{lemma}
\begin{proof} For $0\leq l\leq l_0$,
	by acting $\bpartial^{m-l}$ on the relations of Lemma \ref{lem:relation4}, we get
	$$\bpartial^{m-l}\cF^l_0(L_0,L_1,K_0,K_1,n)=\sum_{k=0}^{m-l}C_{m-k}^l\cF^m_k(L_0,L_1,K_0,K_1,n)\in \R[n+1].$$
	Now the $(l_0+1)\times (l_0+1)$ integer matrix with entries $c_{ij}=C_{m-k_0-i}^j$, $0\leq i,j\leq l_0$ is invertible since the determinant of this matrix is $\pm1$. Let $b_{ij}$ be the entries of the inverse matrix; clearly
	$b_{ij}$ are integers. Let $a_{m,k}= \sum_{l=0}^{l_0}\sum_{j=0}^{l_0}C_{m-k}^l b_{l,j}a_{m,k_0+j}$. These integers satisfy the lemma.
\end{proof}
\begin{proof}[Proof of Lemma \ref{lem:fullrelation0}]
	We only need to show the lemma  when $v_i=u_i=i$, $v'_i=u'_i=h+i$.
	Let $$L_0=\{i_1+1,\dots, h\},\quad L_2=\{h+i_2+1,\dots, s\},$$
	$$K_0=\{j_1+1,\dots, h\},\quad K_2=\{h+j_2+1,\dots, s\}.$$
	By the definition of $\cF^k_l$,
	\begin{eqnarray*}\cF^m_{m-k}(L_0,L_1,K_0,K_1,h)=\sum_{\sigma,\sigma'} \frac{1}{i_1!i_2!j_1!j_2!}\sign(\sigma) \sign(\sigma')\hspace{4cm}\\
		\left(
		\begin{array}{ccc}
			\bpartial^{m-k} (u_h,\dots,u_{i_1+1},\sigma(u_{i_1}),\dots,\sigma(u_1)&|&\sigma'(v_1),\dots, \sigma'(v_{j_1}),v_{j_1+1},\dots,v_h)\\
			\bpartial^{k} (u'_{h'},\dots,u'_{i_2+1},\sigma(u'_{i_2}),\dots,\sigma(u'_1)&|&\sigma'(v'_1),\dots, \sigma'(v'_{j_2}),v'_{j_2+1},\dots,v'_{h'})
		\end{array}
		\right).
	\end{eqnarray*}
	Let $a_k=a_{m,m-k}$, by Lemma \ref{lem:fullrelation}, we have Equation (\ref{eqn: relation4'})
\end{proof}

\section{Proof of Lemma \ref{lem:base0}}\label{section:proof2.5}
In this section we prove Lemma \ref{lem:base0}.
By Lemma \ref{lemma:order1}, we can assume the monomials are expressed as an ordered product $J_{1}J_{2}\cdots J_{b}$ with $J_{a}\prec J_{a+1}$.
For $\alpha \in \cM(\cJ)$,
let $$\R(\alpha)=\{\sum c_i\beta_i\in\R|c_i\in\mathbb Z, \beta_i\in \cM(\cJ),\beta_i\prec \alpha, \beta_i\neq \alpha\},$$
be the space of linear combinations of elements preceding $\alpha$ in $\cM(\cJ)$ with integer coefficients. 
\begin{lemma}\label{lemma:case2}
	If $J_1J_2$ is not standard, $J_1J_2\in \R(J_1)$. 
\end{lemma}
\begin{proof}
	Assume $J_i=\bpartial^{n_i}(u^i_{h_i},\dots,u^i_2,u^i_1|v^i_1,v^i_2,\dots,v^i_{h_i})$, for $i=1,2$. 
	Let $E_1=((u^1_{h_1},n_1),\dots,(u^1_1,0)|(v^1_1,0),\dots,(v^1_{h_1},0))$.
Let $i_0=L(E_1,J_2)$, $j_0=R(E_1,J_2)$ and  $l_0=i_0+j_0$.
 If $i_0\neq 0$, there is $i_0\leq i_1\leq h_2$, such that $u^2_{i_1}<u^1_{i_1-i_0+1}$. If $i_0=0$, let $i_1=0$.
If $j_0\neq 0$,  there is $j_0\leq j_1\leq h$, such that $v^2_{j_1}<v^1_{j_1-j_0+1}$. If $j_0=0$, let $j_1=0$.  Let $i_*=i_1-i_0$, $j_*=j_1-j_0$ and $m=n_1+n_2$.
By Lemma \ref{lem:fullrelation0}, there are integers $a_k$ with $a_{n_2-l}=\delta^0_l$ for $0\leq l\leq l_0-1$.
	 	\begin{equation}\label{eqn:relation110}
	 	\sum \epsilon a_{k} \left(
	 	\begin{array}{ccc}
	 	\bpartial^{m-k} (\underline{u^1_{h_1},\dots,u^1_{i_*+1}},u^1_{i_*},\dots,u^1_1&|&v^1_1,\dots, v^1_{j_*},\underline{v^1_{j_*+1},\dots,v^1_{h_1}})\\
	 	\bpartial^{k} (u^2_{h_2},\dots,u^2_{i_1+1},\underline{u^2_{i_1},\dots,u^2_1}&|&\underline{v^2_1,\dots, v^2_{j_1}},v^2_{j_1+1},\dots,v^2_{h_2})
	 	\end{array}
	 	\right)
	 	\end{equation}
	 $$\in \R[h_1+1].$$
	 		\begin{enumerate}
	 		\item If $h_1=h_2$, then $n_1\leq n_2$.
	 		Since $J_1J_2$ is not standard, $J_2$ is not greater than $E_1$. By Lemma \ref{lemma:critgreat},  $l_0> n_2-n_1\geq 0$.
	 	$J_1J_2\in \R(J_1)$ since in Equation \eqref{eqn:relation110}:
		\begin{itemize}
			\item
			All the terms with $k=n_2$ precede $J_1$ except $J_1J_2$ itself;
			\item
			All the terms with $k=n_2-1,\dots,n_1$ vanish since $a_{k}=0$;
			\item
			All the terms with $k=n_2+1,\dots, m$ precede $J_1$ since the weight of the upper $\bpartial$-list is $m-k<n_1$;
			\item
			All the terms with $k=0,\dots, n_1-1$ precede $J_1$ after exchanging the upper $\bpartial$-list and the lower $\bpartial$-list since the weight of the lower $\bpartial$-list is $k<n_1$;
			\item
			The terms in $\R[h_1+1]$ precede $J_1$ since they have bigger sizes.
		\end{itemize}
		
		\item If $h_1>h_2$.
		Since $J_1J_2$ is not standard, by Lemma \ref{lemma:critgreat}, $l_0> n_2$.
		$J_1J_2\in \R(J_1)$ since in Equation \eqref{eqn:relation110}:
		\begin{itemize}
			\item
			All the terms with $k=n_2$ precede $J_1$ in the lexicographic order except $J_1J_2$ itself;
			\item
			All the terms with $k=n_2-1,\dots,0$ vanish since $a_{k}=0$;
			\item
			All the terms with $k=n_2,\dots, m$ precede $J_1$ since the weight of the upper $\bpartial$-list is $m-k<n_1$;
			\item
			The terms in $\R[h_1+1]$ precede $J_1$ since they have bigger sizes.
		\end{itemize}
	\end{enumerate}
\end{proof}

\begin{proof}[Proof of Lemma \ref{lem:base0}]
	We prove the lemma by induction on $b$.\\
	If $b=1$, $J_1$ is standard.\\
	If $b=2$, by Lemma \ref{lemma:case2}, the lemma is true. \\
	For $b\geq 3$, assume the lemma is true for $b-1$.
 We can assume $J_1\cdots J_{b-1}$ is standard by induction and Lemma \ref{lemma:order2}. Let $E_1\cdots E_{b-1}\in\cS\cM(\cE)$ be the standard order product of element of $\cE$ corresponding to $J_1\cdots J_{b-1}$. If $J_1\cdots J_b$ is not standard, then $J_b$ is not greater than $E_{b-1}$. By Lemma \ref{lemma:straight1} (below), $J_{b-1}J_b=\sum K_if_i$ with $K_i\in \cJ$, $f_i\in \R$ such that $K_i$ is either smaller than $J_{b-1}$ or $K_i$ is not greater than $E_{b-2}$. If $K_i$ is smaller than $J_{b-1}$, then $J_1\cdots J_{b-1}K_if\in \R(J_1\cdots J_{b-1})$. 
  If $K_i$ is not greater than $E_{b-2}$, $J_1\cdots J_{b-2}K_i$ is not standard, so it is in $\R(J_1\cdots J_{b-2})$ by induction. Then $J_1\cdots J_{b-2}K_if\in \R(J_1\cdots J_{b-1})$. So $J_1\cdots J_b=\sum J_1\cdots J_{b-2}K_if_i\in \R(J_1\cdots J_{b-1})$.
\end{proof}
\begin{lemma}\label{lemma:straight1}
	Let $E\in \cE$, $J_a$ and $J_b$ in $\cJ$ with $J_a\prec J_b$, and suppose that $E_a$ is the largest element in $\cE(J_a)$ such that $E\leq E_a$. If $J_b$ is not greater than $E_a$, then $J_aJ_b=\sum K_if_i$ with $K_i\in \cJ$, $f_i\in \R$ such that $K_i$ is either smaller than $J_a$, or $K_i$ is not greater than $E$. 
\end{lemma} 
 \begin{proof}
 	Assume
 	$$J_a=\bpartial^{n_a}(u'_{h_a},\dots,u'_{1}|v'_{1},\dots,v'_{h_a}), \quad J_b=\bpartial^{n_b}(u^b_{h_b},\dots,u^b_{1}|v^b_{1},\dots,v^b_{h_b}).$$
 	$J_a\prec J_b$, so $h_a\geq h_b$.
 	Assume
$||E_a(h_b)||=\bpartial^{m_a}(u^a_{h_b},\dots,u^a_{1}|v^a_{1},\dots,v^a_{h_b})$.
Let  $m=n_b+n_a$.
	Let $i_0=L(E_a, J_b)$, $j_0=(E_a, J_b)$ and $l_0=i_0+j_0$.
 If $i_0\neq 0$, there is $i_0\leq i_1\leq h_b$, such that $u^b_{i_1}<u^a_{i_1-i_0+1}$. If $i_0=0$, let $i_1=0$.
 If $j_0\neq 0$, there is $j_0\leq  j_1\leq h_b$, such that $v^b_{j_1}<v^a_{j_1-j_0+1}$.
If $j_0=0$, let $j_1=0$.
Since $J_b$ is not greater than $ E_a$, by Lemma \ref{lemma:critgreat},
 \begin{equation}\label{eqn:numberl0} l_0=i_0+j_0> n_b-m_a.
 \end{equation}
 By definition, $\{u^a_{h_b},\dots,u^a_{1}\}$ is a subset of $\{u'_{h_a},\dots, u'_{1}\}$ with
 $u'_{i}<u'_{i+1}$ and $u^a_{i}<u^a_{i+1}$. If we assume $u'_{i_2}=u^a_{i_1-i_0+1}$, we have $i_2\geq i_1-i_0+1$.
Similarly, $\{v^a_{h_b},\dots,v^a_{1}\}$ is a subset of $\{v'_{h_a},\dots, v'_{1}\}$ with
$v'_{j}<v'_{j+1}$ and $v^a_{i}<v^a_{j+1}$; if we assume $v'_{j_2}=v^a_{j_1-j_0+1}$, then $j_2\geq j_1-j_0+1$.

Now we prove the lemma. The proof is quite long and it is divided into three cases.
	
	\textsl{Case 1:} $h_a=h_b$. Let $a_{n_b-l}=\delta^0_l$ for $0\leq l\leq l_0-1$. 
	By Lemma \ref{lem:fullrelation0}, there are integers $a_k$, such that 
	\begin{equation}\label{eqn:relation111}
	\sum \epsilon a_{k} \left(
	\begin{array}{ccc}
	\bpartial^{m-k} (\underline{u'_{h_a},\dots,u'_{i_1-i_0+1}},u'_{i_1-i_0},\dots,u'_1&|&v'_1,\dots, v'_{j_1-j_0},\underline{v'_{j_1-j_0+1},\dots,v'_{h_a}})\\
	\bpartial^{k} (u^b_{h_b},\dots,u^b_{i_1+1},\underline{u^b_{i_1},\dots,u^b_1}&|&\underline{v^b_1,\dots, v^b_{j_1}},v^b_{j_1+1},\dots,v^b_{h_b})
	\end{array}
	\right)
	\end{equation}
 $$\in \R[h_1+1].$$
	$J_aJ_b\in \R(J_a)$ since in the above equation, 
	\begin{itemize}
		\item
		All the terms with $k=n_b$ precede $J_a$ in the lexicographic order except $J_aJ_b$ itself;
		\item
		All the terms with $k=n_b-1,\dots,n_a$ vanish since $n_a=m_a$, $a_{k}=0$;
		\item
		All the terms with $k=n_b+1,\dots, m$ precede $J_a$ since the weight of the upper $\bpartial$-list is $m-k<n_a$;
		\item
		All the terms with $k=0,\dots, n_a-1$ precede $J_a$ after exchanging the upper $\bpartial$-list and the lower $\bpartial$-list since  the weight of the lower $\bpartial$-list is $k<n_a$;
		\item
		The terms in $\R[h+1]$ precede $J_a$ since they have bigger sizes.
	\end{itemize}

	\textsl{Case 2:} $h_b<h_a$ and $n_b<m_a$. 
	
	By Lemma \ref{lem:fullrelation0}, $$\sum_{\substack{0\leq i,j\leq h_b\\i+j<2h_b}}\sum_{\sigma, \sigma'}\frac {(-1)^{i+j}\sign(\sigma)\sign(\sigma')}{i!(h_b-i)!j!(h_b-j)!}\sum \epsilon a^{i,j}_{k} $$
	\begin{equation}\label{eqn:relationterms} \left(
	\begin{array}{ccc}
	\bpartial^{m-k}(\underline{u'_{h_a},\dots,u'_{i+1}},u^b_{\sigma{(i)}}\dots,u_{\sigma(2)}^b,u^b_{\sigma(1)}&|&v^b_{\sigma'(1)},v^b_{\sigma'(2)},\dots,v^b_{\sigma(j)},\underline{v'_{j+1},\dots,v'_{h_a}})\\
	\bpartial^k(\underline{u_{\sigma(h_b)}^b,\dots,u^b_{\sigma{(i+1)}},u'_{i}\dots,u'_{1}}&|&\underline{v'_{1},\dots,v'_{j}, v^b_{\sigma'(j+1)},\dots,v^b_{\sigma'(h_b)}})
	\end{array}
	\right)
	\end{equation}
	$$  \in \R[h_a+1].$$
	Here  $a_k^{i,j}$ are integers and $a_{n_b-l}^{i,j}=\delta_{0,l}$ for $0 \leq l<2h_b -(i+j)$. The second summation is over all pairs of permutations $\sigma$ and $\sigma'$ of $\{1,\dots,h_b\}$.
	In the above equation:
	\begin{itemize}
		\item
		The terms in $\R[h_a+1]$ precede $J_a$ since they have bigger sizes.
		\item
		All the terms with $k=n_b,\dots, m$ precede $J_a$ since the weight of the upper $\bpartial$-list is $m-k<n_a$.
		\item
		The terms with $k=n_b$ are $J_aJ_b$
		and the terms with the lower $\bpartial$-lists  $$K_0=\bpartial^{n_b}(u'_{i_{h_b}},\dots,u'_{i_2},u'_{i_1}|v'_{j_1},v'_{j_2},\dots,v'_{j_{h_b}})\in \cJ.$$
		All the other terms cancel.
		By Corollary \ref{cor:lrnumber1} and \ref{cor:lrnumber2}, 
		$$L(E_{b-2},K_0)+R(E_{b-2},K_0)\geq L(E_{b-2},E_{a}(h_b))+R(E_{b-2},E_{a}(h_b))$$
		$$=m_a-m_{b-2}>n_b-m_{b-2}.$$
		By Lemma \ref{lemma:critgreat}, $K_0$ is not greater than $E_{b-2}$.
		\item
		The terms with $k<n_b$ vanish unless $2h_b-(i+j)\leq  n_b-k$. In this case, the lower $\bpartial$-lists of the terms are
		$$K_1=\bpartial^k(u_{\sigma(h_b)}^b,\dots,u^b_{\sigma{(i+1)}},u'_{s_i}\dots,u'_{s_1}|v'_{t_1},\dots,v'_{t_j}, v^b_{\sigma(j+1)},\dots,v^b_{\sigma'(h_b)}).$$
	    By Lemma \ref{lemma:replace},
		$$L(E_{b-2},K_1)\geq L(E_{b-2},K_0)-(h_b-i),\quad R(E_{b-2},K_1)\geq R(E_{b-2},K_0)-(h_b-j).$$
		So
		$$L(E_{b-2},K_1)+R(E_{b-2},K_1)>n_b-m_{b-2}-(2h_b-i-j) \geq k-m_{b-2}.$$	
		By Lemma \ref{lemma:critgreat}, $K_1$ is not greater than $E_{b-2}$. 
	\end{itemize}

	\textsl{Case 3:}
	$h_a>h_b$ and $n_b\geq m_a$.
	\begin{enumerate}
		\item  If $i_0=0$, then $j_0> 0$. Since if $j_0=0$, then $J_b$ is greater than $E_a$ and $J_1\cdots J_b$ is standard.
		
		By Equation \eqref{lem:fullrelation0}, 
		\begin{equation}\label{eqn:relationterms30}\sum_{0\leq i\leq h_b}\sum_{\substack{ j_1\geq t\\ j_2>t\\j_1+h_b\neq i+t}}\sum_{\sigma, \sigma'}\frac {(-1)^{i+t}\sign(\sigma)\sign(\sigma')}{i!(h_b-i)!t!(j_2-1-t)!}\sum_{k=0}^{m} \epsilon a^{i,t}_{k}
		\end{equation}
$$
		\bpartial^{m-k}(\underline{u_{h_a}',\dots,u_{i+1}'},u^b_{\sigma{(i)}} ,\dots,u_{\sigma(2)}^b,u^b_{\sigma(1)}|\underline{v^b_{1} ,\dots,v^b_{t} ,v'_{\sigma'(t+1)},\dots, v'_{\sigma'(j_2-1)},v'_{j_2},\dots,v'_{h}})$$
		$$\hspace{1.5cm}
		\bpartial^k(\underline{u_{\sigma(h_b)}^b,\dots,u^b_{\sigma{(i+1)}},u'_{i} , \dots,u'_{1}}|v'_{\sigma'(1)},\dots,v'_{\sigma'(t)},\underline{v^b_{t+1},\dots,v^b_{j_1}},v^b_{j_1+1},\dots,v^b_{h_b}  )
$$
		$$  \in \R[h_a+1].$$
		Here  $a_k^{i,t}$ are integers with $a_{n_b-l}^{i,t}=\delta_{0,l}$
		for $0 \leq l<h_b-i+j_1-t$, $\sigma$ are permutations of $\{1,\dots,h_b\}$  
		and $\sigma'$ are permutations of $\{1,\dots,j_2-1\}$. Next,
		\begin{equation}\label{eqn:relationterms31}\sum_{\substack{0\leq i\leq h_b\\ j_1\leq j\leq h_b}}\sum_{\substack{ j_1\geq t\\ j_2>t\\h_b+j_1\neq i+t}}\sum_{\sigma, \sigma',\sigma_2}\frac {(-1)^{i+j+t}\sign(\sigma)\sign(\sigma')\sign(\sigma_2)}{i!(h_b-i)!(j-j_1)!(h_b-j)!t!(j_2-1-t)!}\sum_{k=0}^{m}\epsilon a^{i,j,t}_{k} 
			\end{equation}
	
$$\bpartial^{m-k}(\underline{u_{h_a}',\dots,u_{i+1}'},u^b_{\sigma{(i)}} ,\dots,u_{\sigma(2)}^b,u^b_{\sigma(1)}\hspace{5cm}$$
$$\hspace{3cm}	|\underline{v'_{h_b+1} ,\dots, v'_{h_a},v^b_1, \dots, v^b_{j_1}, v^b_{\sigma'(j_1+1)}, \dots, v^b_{\sigma'(j)}} ,\dots, v^b_{\sigma'(h_b)})$$
		$$
		\bpartial^k(\underline{u_{\sigma(h_b)}^b,\dots,u^b_{\sigma{(i+1)}},u'_{i} ,\dots,u'_{1}} \hspace{6cm}$$
		$$\hspace{3cm}|v'_{\sigma_2(1)} ,\dots , v'_{\sigma_2(t)},\underline{v'_{\sigma_2(t+1)}, \dots, v'_{\sigma_2(j_2-1)}, v'_{j_2}, \dots , v'_{h_b+1}})
$$
		$$  \in \R[h_a+1].$$
		Here  $a_k^{i,j,t}$ are integers and $a_{n_b-l}^{i,j,t}=\delta_{0,l}$ for $0 \leq l<h_b -i+j-t$, $\sigma$ are permutations of $\{1,\dots,h_b\}$, $\sigma'$ are permutations of $\{h_b,\dots,j_1+1\}$,
		and $\sigma_2$ are permutations of $\{1,\dots, j_2-1\}$.
		
		We use Equation \eqref{eqn:relationterms30} if $j_2>j_1-j_0+1$ and use 
		Equation \eqref{eqn:relationterms31} if $j_2=j_1-j_0+1$.
		In the above equations:
		\begin{enumerate}
			\item
			The terms in $\R[h_a+1]$ precede $J_a$ since they have bigger sizes;
			\item
			All the terms with $k=n_b+1,\dots, m$ precede $J_a$ since the weight of upper $\bpartial$-list is $m-k<n_a$.
			\item
			The terms with $k=n_b$ are $J_aJ_b$, the terms with upper $\bpartial$-list preceding $J_a$ (the upper $\bpartial$-lists are the $\bpartial$-lists given
			by replacing some $v'_{j}$, $j\geq j_2$ in $J_a$ by some $u^b_{k}$, $k\leq j_1$),
			and the terms with the lower $\bpartial$-lists
			$$K_0=\bpartial^k(u'_{i_{h_b}},\dots,u'_{i_1}|v'_{\sigma(1)},\dots,v'_{\sigma(j_1)},v^b_{j_1+1},\dots,v^b_{h_b}).$$
			All of the other terms cancel.
			If $j_2-1<j_1$, the terms of the form $K_0$ do not appear in Equations \eqref{eqn:relationterms30}, \eqref{eqn:relationterms31}.
			Otherwise, $j_2-1\geq j_1$. In this case $j_2\geq j_1+1>j_1-j_0+1$.
			By Corollary \ref{cor:lrnumber2}, 
			\begin{equation}\label{eqn:LK0}
			L(E_{b-2},||E_a(h_b)||)\leq L(E_{b-2},K_0).
		\end{equation}
	 By Lemma \ref{lem:lrnumber},  
	 \begin{equation}\label{eqn:RK0}
	R(E_{b-2},K_0)> R(E_{b-2},||E_a(h_b))+j_1-(j_1-j_0+1).
	\end{equation}
			By Corollary \ref{cor:lrnumber1},
		\begin{equation}\label{eqn:lr300}
		L(E_{b-2},||E_a(h_b)||)+R(E_{b-2},||E_a(h_b)||)=wt(E_{a}(h_b))-wt(E_{b-2}(h_b)).
		\end{equation}
			 So by Equations \eqref{eqn:LK0}, \eqref{eqn:RK0}, \eqref{eqn:lr300}, and \eqref{eqn:numberl0},
		$$
			L(E_{b-2},K_0)+R(E_{b-2},K_0)\geq m_a-m_{b-2}+j_0>n_b-m_{b-2}.		
	$$
	By Lemma \ref{lemma:critgreat}, $K_0$ is not greater than $E_{b-2}$.
		
			\item If $j_2>j_1-j_0+1$,
			the terms with $k<n_b$ in Equation \eqref{eqn:relationterms30} vanish unless $h_b-i+j_1-t\leq  n_b-k$.  The lower $\bpartial$-lists of these terms are
			$$K_1= \bpartial^k(\underline{u_{\sigma(h_b)}^b,\dots,u^b_{\sigma{(i+1)}},u'_{i}\dots,u'_{1}}|v'_{\sigma'(1)},\dots,v'_{\sigma'(t)},\underline{v^b_{t+1},\dots,v^b_{j_1}},v^b_{j_1+1},\dots,v^b_{h_b}  )$$
			Underlined $u$ and $v$ can be any underlined $u$ and $v$ in Equation \eqref{eqn:relationterms30}, respectively.

           By Lemma \ref{lemma:replace} and Equation \eqref{eqn:LK0}, 
          \begin{equation}\label{eqn:LK1}
       L(E_{b-2}, K_1)\geq L(E_{b-2},K_0)-(h_b-i)\geq L(E_{b-2},||E_a(h_b)||)-(h_b-i).   
    \end{equation}	
           		
			Since $j_2>j_1-j_0+1$, 
		 by Lemma \ref{lem:lrnumber},
			\begin{equation}\label{eqn:RK1}
		R(E_{b-2},||E_a(h_b)||)+t-(j_1-j_0+1).
			\end{equation}
		So by Equation \eqref{eqn:lr300}, \eqref{eqn:LK1}, \eqref{eqn:RK1}, and \eqref{eqn:numberl0},
			$$L(E_{b-2}, K_1)+R(E_{b-2}, K_1)\geq m_a-m_{b-2}-(h_b-i)+t-(j_1-j_0)>k-m_{b-2}. $$
		By Lemma \ref{lemma:critgreat}, $K_1$ is not greater than $E_{b-2}$.

			\item If $j_2=j_1-j_0+1$, the terms with $k<n_b$ in Equation \eqref{eqn:relationterms31} vanish unless $h_b-i+j-t  \leq  n_b-k$. In this case, the lower $\bpartial$-lists of the terms are
			$$K_1=   \bpartial^k(\underline{u_{\sigma(h_b)}^b,\dots,u^b_{\sigma{(i+1)}},u'_{i},\dots,u'_{1}}\hspace{6cm}
			$$
			$$\hspace{4cm}|
			v'_{\sigma_2(1)} ,\dots ,v'_{\sigma_2(t)},\underline{v'_{\sigma_2(t+1)} ,\dots , v'_{\sigma_2(j_2-1)}, v'_{j_2} ,\dots , v'_{h_b+1}})$$
			Underlined $u$ and $v$ can be any underlined $u$ and $v$ in Equation \eqref{eqn:relationterms31}.

				Let 
			$$
			K'_1=\bpartial^k(\underline{u_{\sigma(h_b)}^b,\dots,u^b_{\sigma{(i+1)}},u'_{i}, \dots,u'_{1}}|
			v'_{\sigma_2(1)}, \dots, v'_{\sigma_2(j_2-1)},v'_{l_{j_2}},\dots, v'_{l_{h_b}})$$
			Here $j_2\leq l_{j_2}< l_{j_2+1}<\cdots< l_{h_b}\leq h_a$. By Lemma \ref{lemma:replace} and Equation \eqref{eqn:LK0},
			 \begin{equation}\label{eqn:LK1-1}
			L(E_{b-2}, K_1)\geq L(E_{b-2},K_0)-(h_b-i)\geq L(E_{b-2},||E_a(h_b)||)-(h_b-i).   
			\end{equation}	
			By Lemma \ref{lemma:replace}, there is some $K'_1$ such that
					\begin{equation}\label{eqn:RK1-1}
					R(E_{b-2},K_1)\geq R(E_{b-2},K'_1)-(j_2-1-t)-(j-j_1)
					\end{equation}
	 since in $K_1$ the number of $v^b_j$ with $v^b_j>v'_{j_2}$ is at most $j-j_1$. By Corollary \ref{cor:lrnumber2},
	 \begin{equation}\label{eqn:RK1-1+}
	 R(E_{b-2},K_1')\geq R(E_{b-2},||E_a(h_b)||).
	 \end{equation}
	  So by Equations \eqref{eqn:LK1-1}, \eqref{eqn:RK1-1}, \eqref{eqn:RK1-1+}, \eqref{eqn:lr300}, and \eqref{eqn:numberl0},
			\begin{eqnarray*}
			L(E_{b-2},K_1)+R(E_{b-2},K_1)&\geq& m_a-m_{b-2}-(h_b-i)-(j_2-1-t)-(j-j_1)\\
			&>& k-wt(E_{b-2}(h_b)).
		\end{eqnarray*}	
				By Lemma \ref{lemma:critgreat}, $K_1$ is not greater than $E_{b-2}$.
		
		\end{enumerate}

		\item $j_0=0$, then $i_0> 0$. The proof is similar to the case of $i_0=0$.
		\item
		$i_0> 0$ and $j_0> 0$.
		By Lemma \ref{lem:fullrelation0}, we have
	\begin{equation}\label{eqn:relationterms51} \sum_{\substack{ i_1\geq s\\ i_2>s}}\sum_{\substack{ j_1\geq t\\ j_2>t\\j_1+i_1\neq s+t}}\sum_{\sigma, \sigma'}\frac {(-1)^{s+t}\sign(\sigma)\sign(\sigma')}{s!(i_2-1-s)!t!(j_2-1-t)!}\sum_{k=0}^{m} \epsilon a^{s,t}_{k} 
		\end{equation}
	$$
		\bpartial^{m-k}(\underline{u'_{h} ,\dots,u'_{i_2} ,u'_{\sigma(i_2-1)},\dots, u'_{\sigma(s+1)},u_{s}^b,\dots,u^b_{1}}|\underline{v^b_{1} ,\dots,v^b_{t} ,v'_{\sigma'(t+1)},\dots, v'_{\sigma'(j_2-1)},v'_{j_2},\dots,v'_{h}})
		$$
		$$
		\bpartial^k(u^b_{h_b},\dots,u^b_{i_1+1},\underline{u^b_{i_1},\dots,u^b_{s+1}},u'_{\sigma(s)},\dots,u'_{\sigma(1)}|v'_{\sigma'(1)},\dots,v'_{\sigma'(t)},\underline{v^b_{t+1},\dots,v^b_{j_1}},v^b_{j_1+1},\dots,v^b_{h_b})
	$$
		$$  \in \R[h_a+1].$$
		Here $a_k^{s,t}$ are integers and $a_{n_b-l}^{s,t}=\delta_{0,l}$
		for $0 \leq l<i_1-s+j_1-t$, $\sigma$ are permutations of $\{1,\dots,i_2-1\}$,  
		and $\sigma'$ are permutations of $\{1,\dots,j_2-1\}$. 
		
	\begin{equation}\label{eqn:relationterms52}\sum_{\substack{ i_1\geq s\\ i_2>s}}\sum_{\substack{j\geq j_1\geq t\\ j_2>t\\ i_1+j_1>s+t}}\sum_{\sigma, \sigma',\sigma_2}\frac {(-1)^{j+s+t}\sign(\sigma)\sign(\sigma')\sign(\sigma_2)}{(j-j_1)!(h_b-j)!s!(i_2-1-s)!t!(j_2-1-t)!}\sum_{k=0}^{m}\epsilon a^{j,s,t}_{k} \end{equation}
$$\bpartial^{n-k}(\underline{u'_{h} ,\dots, u'_{i_2} ,u'_{\sigma(i_2-1)},\dots, u'_{\sigma(s+1)},u_{s}^b,\dots, u^b_{1}}\hspace{4cm}$$
$$\hspace{4cm}|\underline{v'_{h_b+1},\dots, v'_{h_a},v^b_1,\dots, v^b_{j_1}, v^b_{\sigma'(j_1+1)},\dots, v^b_{\sigma'(j)}}, \dots , v^b_{\sigma'(h_b)})
$$
$$
		\bpartial^k(u^b_{h_b},\dots,u^b_{i_1+1},\underline{u^b_{i_1},\dots,u^b_{s+1}},u'_{\sigma(s)},\dots,u'_{\sigma(1)}
		\hspace{4cm}$$
		$$\hspace{4cm}|
		v'_{\sigma_2(1)}, \dots, v'_{\sigma_2(t)},\underline{v'_{\sigma_2(t+1)},\dots, v'_{\sigma_2(j_2-1)}, v'_{j_2},\dots, v'_{h_b+1}})
$$
		$$  \in \R[h_a+1].$$
		Here  $a_k^{j,s,t}$ are integers and $a_{n_b-l}^{i,j,s,t}=\delta_{0,l}$ for $0 \leq l<(i_1+j-s-t)$,
		$\sigma'$ are permutations of $\{h_b,\dots,j_1+1\}$,
		$\sigma$ are permutations of $\{1,\dots,i_2-1\}$, and $\sigma_2$ are permutations of $\{1,\dots, j_2-1\}$.
		
			\begin{equation}\label{eqn:relationterms53} \sum_{\substack{i\geq i_1\geq s\\ i_2>s}}\sum_{\substack{j_1\geq t\\ j_2>t\\ i_1+j_1>s+t}}\sum_{\sigma, \sigma',\sigma_1}\frac {(-1)^{i+s+t}\sign(\sigma)\sign(\sigma')\sign(\sigma_1)}{(i-i_1)!(h_b-i)!s!(i_2-1-s)!t!(j_2-1-t)!}\sum_{k=0}^{m}\epsilon a^{i,s,t}_{k} 
			\end{equation}
$$
		\bpartial^{n-k}(u^b_{\sigma(h_b)},\dots,\underline{u^b_{\sigma(i)},\dots, u^b_{\sigma(i_1+1)},u^b_{i_1}, \dots, u^b_{1},u'_{h_a},\dots,  u'_{h_b+1}}
		\hspace{4cm}$$
		$$	\hspace{4cm}
		|\underline{v^b_{1} ,\dots,v^b_{t} ,v'_{\sigma'(t+1)},\dots, v'_{\sigma'(j_2-1)},v'_{j_2},\dots,v'_{h}})
		$$
		$$
		\bpartial^k(\underline{u'_{h_b+1},\dots, u'_{i_2},u'_{\sigma_1(i_2-1)}, \dots, u'_{\sigma_1(s+1)}},u'_{\sigma_1(s)}, \dots, u'_{\sigma_1(1)}	\hspace{4cm}
		$$
		$$	\hspace{4cm}|v'_{\sigma'(1)},\dots,v'_{\sigma'(t)},\underline{v^b_{t+1},\dots,v^b_{j_1}},v^b_{j_1+1},\dots,v^b_{h_b}  )
	$$
		$$  \in \R[h_a+1].$$
		Here  $a_k^{i,s,t}$ are integers and $a_{n_b-l}^{i,s,t}=\delta_{0,l}$ for $0 \leq l<(i+j_1-s-t)$,
		$\sigma$ are permutations of $\{h_b,\dots, i_1+1\}$, 
		$\sigma_1$ are permutations of $\{1,\dots,i_2-1\}$, and $\sigma'$ are permutations of $\{1,\dots, j_2-1\}$.

		\begin{equation}\label{eqn:relationterms54}\sum_{\substack{i\geq i_1\geq s\\ i_2>s}}\sum_{\substack{j\geq j_1\geq t\\ j_2>t\\ i_1+j_1>s+t}}\sum_{\sigma, \sigma',\sigma_1,\sigma_2}\frac {(-1)^{i+j+s+t}\sign(\sigma)\sign(\sigma')\sign(\sigma_1)\sign(\sigma_2)}{(i-i_1)!(h_b-i)!(j-j_1)!(h_b-j)!s!(i_2-1-s)!t!(j_2-1-t)!}\sum_{k=0}^{m}\epsilon a^{i,j,s,t}_{k} 
		\end{equation}
$$
		\bpartial^{n-k}(u^b_{\sigma(h_b)},\dots,\underline{u^b_{\sigma(i)},\dots, u^b_{\sigma(i_1+1)},u^b_{i_1}, \dots, u^b_{1},u'_{h_a},\dots,  u'_{h_b+1}}\hspace{4cm}$$
		$$\hspace{4cm}
		|\underline{v'_{h_b+1},\dots, v'_{h_a},v^b_1,\dots, v^b_{j_1}, v^b_{\sigma'(j_1+1)},\dots, v^b_{\sigma'(j)}},\dots,  v^b_{\sigma'(h_b)})
		$$
		$$
		\bpartial^k(\underline{u'_{h_b+1},\dots, u'_{i_2},u'_{\sigma_1(i_2-1)}, \dots, u'_{\sigma_1(s+1)}},u'_{\sigma_1(s)} ,\dots, u'_{\sigma_1(1)}\hspace{4cm}$$
		$$\hspace{4cm}|v'_{\sigma_2(1)}, \dots, v'_{\sigma_2(t)},\underline{v'_{\sigma_2(t+1)},\dots, v'_{\sigma_2(j_2-1)}, v'_{j_2},\dots, v'_{h_b+1}})
$$
		$$  \in \R[h_a+1].$$
		Here  $a_k^{i,j,s,t}$ are integers and $a_{n_b-l}^{i,j,s,t}=\delta_{0,l}$ for $0 \leq l<(i+j-s-t)$,
		$\sigma$ are permutations of $\{h_b,\dots, i_1+1\}$, $\sigma'$ are permutations of $\{h_b,\dots,j_1+1\}$,
		$\sigma_1$ are permutations of $\{1,\dots,i_2-1\}$, and $\sigma_2$ are permutations of $\{1,\dots, j_2-1\}$.
		\begin{itemize}
			\item We use Equation (\ref{eqn:relationterms51}) when $i_2>i_1-i_0+1$ and $j_2>j_1-j_0+1$;
			\item We use Equation (\ref{eqn:relationterms52}) when $i_2>i_1-i_0+1$ and $j_2=j_1-j_0+1$;
			\item We use Equation (\ref{eqn:relationterms53}) when $i_2=i_1-i_0+1$ and $j_2>j_1-j_0+1$;
			\item We use Equation (\ref{eqn:relationterms54}) when $i_2=i_1-i_0+1$ and $j_2=j_1-j_0+1$.
		\end{itemize}
		In the above relations:
		\begin{enumerate}
			\item
			The terms in $\R[h_a+1]$ precede $J_a$ since they have bigger sizes.
			\item
			All the terms with $k=n_b+1,\dots, n$ precede $J_a$ since the weight of the upper $\bpartial$-list is $n-k<n_a$.
			\item
			The terms with $k=n_b$ are $J_aJ_b$, the terms with the upper $\bpartial$-list preceding $J_a$ (the upper $\bpartial$-lists are the $\bpartial$-lists given by replacing some $u'_{i}$, $i\geq i_2$ in $J_a$ by some $u^b_{k}$, $k\leq i_1$ or $v'_j$, $j\geq j_2$ by $v^b_{k}$, $k\leq j_1$), 
			and the terms with the lower $\bpartial$-lists
			$$K_0=\bpartial^{n_b}(u_{h_b}^b,\dots,u_{i_1+1}^b, u'_{\sigma_1(i_1)},\dots, u'_{\sigma_1(1)}|v'_{\sigma_2(1)},\dots, v'_{\sigma_2(j_1)},v_{j_1+1}^b,\dots,u_{h_b}^b).$$
			All of the other terms cancel.
			By Lemma \ref{lem:lrnumber}, 
			$$L(E_{b-2},K_0)>L(E_{b-2},||E_a(h)||)+i_1-(i_1-i_0+1);$$
			$$R(E_{b-2},K_0)>R(E_{b-2},||E_a||(h))+j_1-(j_1-j_0+1).$$
			By the above two inequalities,
			\begin{eqnarray*}L(E_{b-2},K_0)+R(E_{b-2},K_0)&\geq & L(E_{b-2},||E_a(h_b)||)+R(E_{b-2},||E_a(h_b)||)+i_0+j_0\\
				(\text{by Corollary \ref{cor:lrnumber1}})\quad\quad
				&=&wt(E_a(h_b))-wt(E_{b-2})+i_0+j_0\\
			(\text{by Equation \eqref{eqn:numberl0}}\quad\quad	&>&wt(E_b)-wt(E_{b-2}(h_b)).
			\end{eqnarray*}
						
			By Lemma \ref{lemma:critgreat}, $K_0$ is not greater than $E_{b-2}$.
	
			\item  When $i_2>i_1-i_0+1$ and $j_2>j_1-j_0+1$;
			The terms with $k<n_b$ in Equation \eqref{eqn:relationterms51} vanishes  unless $i_1+j_1-s-t\leq  n_b-k$. In this case, the lower $\bpartial$-lists of the terms are
			$$K_1=\bpartial^k(u^b_{h_b},\dots,u^b_{i_1+1},\underline{u^b_{i_1},\dots,u^b_{s+1}},u'_{\sigma(s)},\dots,u'_{\sigma(1)}$$
			$$\hspace{4cm}|v'_{\sigma'(1)},\dots,v'_{\sigma'(t)},\underline{v^b_{t+1},\dots,v^b_{j_1}},v^b_{j_1+1},\dots,v^b_{h_b}  )$$
			The underlined $u$ and $v$ in $K_1$ can be any underlined $u$ and $v$ in Equation \eqref{eqn:relationterms51}. By Lemma \ref{lem:lrnumber},
			$$L(E_{b-2},K_1)> L(E_{b-2},E_a(h_b))+s-(i_1-i_0+1);$$
			$$R(E_{b-2},K_1)> R(E_{b-2},E_a(h_b))+t-(j_1-j_0+1);$$
			\begin{eqnarray*}
				L(E_{b-2},K_1)+R(E_{b-2},K_1)&\geq& L(E_{b-2},E_a(h_b))+R(E_{b-2},E_a(h_b))\\
				&&+k-n_b+i_0+j_1\\
				(\text{by Corollary \ref{cor:lrnumber1}}) &=&wt(E_a(h_b))-wt(E_{b-2}(h_b))+k-n_b+i_0+j_0\\
			(\text{by Equation \eqref{eqn:numberl0}}\,	&>& k-wt(E_{b-2}(h_b))
			\end{eqnarray*}
			By Lemma \ref{lemma:critgreat}, $K_1$ is not greater than $E_{b-2}$.
	
			\item \label{enum:itemd} When $i_2>i_1-i_0+1$ and $j_2=j_1-j_0+1$, the terms with $k<n_b$ are the terms in Equation \eqref{eqn:relationterms52}, such that $i_1+j-s-t\leq  n_b-k$. 			
			$$
			K_1=\bpartial^k(u^b_{h_b},\dots,u^b_{i_1+1},\underline{u^b_{i_1},\dots,u^b_{s+1}},u'_{\sigma(s)},\dots,u'_{\sigma(1)}$$
			$$\hspace{4cm} |
			v'_{\sigma_2(1)} ,\dots, v'_{\sigma_2(t)},\underline{v'_{\sigma_2(t+1)},\dots, v'_{\sigma_2(j_2-1)}, v'_{j_2},\dots, v'_{h_b+1}})$$
			The underlined $u$ and $v$ in $K_1$ can be any underlined $u$ and $v$ in Equation \eqref{eqn:relationterms51}. By Lemma \ref{lem:lrnumber}
			\begin{equation}\label{eqn:LK1-2}
			L(E_{b-2},K_1)> L(E_{b-2},||E_a(h_b)||)+s-(i_1-i_0+1).
			\end{equation}
			Let 
			$$
			K'_1=\bpartial^k(u^b_{h_b},\dots,u^b_{i_1+1},\underline{u^b_{i_1},\dots,u^b_{s+1}},u'_{\sigma(s)},\dots,u'_{\sigma(1)}$$
			$$\hspace{4cm}|
			v'_{\sigma_2(1)} ,\dots ,v'_{\sigma_2(j_2-1)},v'_{l_{j_2}},\dots, v'_{l_{h_b}})$$
			Here $j_2\leq l_{j_2}< l_{j_2+1}<\cdots< l_{h_b}\leq h_a$. By Lemma \ref{lemma:replace}, there is some $K'_1$ such that
			\begin{equation}\label{eqn:RK1-2}
			R(E_{b-2},K_1)\geq R(E_{b-2},K'_1)-(j_2-1-t)-(j-j_1)
			\end{equation}
		 since in $K_1$ the number of $v^b_j$ with $v^b_j>v'_{j_2}$ is at most $j-j_1$. By Corollary \ref{cor:lrnumber2}, \begin{equation}\label{eqn:RK1-2+}
		 	R(E_{b-2},K_1')\geq R(E_{b-2},||E_a(h_b)||).
		 	\end{equation}
		 	 So by Equation \eqref{eqn:LK1-2}, \eqref{eqn:RK1-2}, and \eqref{eqn:RK1-2+},
						\begin{eqnarray*}
				L(E_{b-2},K_1)+R(E_{b-2},K_1)&\geq& L(E_{b-2},||E_a(h_b)||)+R(E_{b-2},||E_a(h_b)||)\\
				& &+s+t-(i_1-i_0)-(j-j_0)\\
				(\text{by Corollary \ref{cor:lrnumber1}})&\geq& wt(E_a(h_b))-wt(E_{b-2}(h_b))+k-n_b+i_0+j_0\\
			(\text{by Equation \eqref{eqn:numberl0}})	&>& k-wt(E_{b-2}(h_b)).
			\end{eqnarray*}
			
			By Lemma \ref{lemma:critgreat}, $K_1$ is not greater than $E_{b-2}$.
			\item  When $i_2=i_1-i_0+1$ and $j_2>j_1-j_0+1$ and $k<n_b$, the proof is similar to the proof of case \eqref{enum:itemd}.
			
			\item  When $i_2=i_1-i_0+1$ and $j_2=j_1-j_0+1$, the terms with $k<n_b$ are the terms in Equation \eqref{eqn:relationterms54}, such that $i+j-s-t\leq  n_b-k$. In this case, the lower $\bpartial$-lists of the terms are
			$$K_1=\bpartial^k(\underline{u'_{h_b+1},\dots, u'_{i_2},u'_{\sigma_1(i_2-1)}, \dots, u'_{\sigma_1(s+1)}},u'_{\sigma_1(s)}, \dots, u'_{\sigma_1(1)} $$
			$$\hspace{4cm}|
			v'_{\sigma_2(1)} ,\dots, v'_{\sigma_2(t)},\underline{v'_{\sigma_2(t+1)},\dots, v'_{\sigma_2(j_2-1)}, v'_{j_2},\dots, v'_{h_b+1}})$$
			The underlined $u$ and $v$ in $K_1$ can be any underlined $u$ and $v$ in Equation \eqref{eqn:relationterms54}.
			Let 
			$$
			K'_1=\bpartial^k(u'_{k_{h_b}},\dots,
			u'_{k_{i_2}},u'_{\sigma(i_2-1)},\dots,u'_{\sigma(1)}|
			v'_{\sigma_2(1)}, \dots ,v'_{\sigma_2(j_2-1)},v'_{l_{j_2}},\dots, v'_{l_{h_b}}).$$
			Here $i_2\leq k_{i_2}< k_{i_2+1}<\cdots< k_{h_b}\leq h_a$ and $j_2\leq l_{j_2}< l_{j_2+1}<\cdots< l_{h_b}\leq h_a$. By Lemma \ref{lemma:replace}, there is some $K'_1$ such that 
			\begin{eqnarray}\label{eqn:LK1-3}
			R(E_{b-2},K_1)&\geq& R(E_{b-2},K'_1)-(j_2-1-t)-(j-j_1);\\
			\label{eqn:RK1-3}
			L(E_{b-2},K_1)&\geq& L(E_{b-2},K'_1)-(i_2-1-s)-(i-i_1).
			\end{eqnarray}
			since in $K_1$ the number of $v^b_l$ with $v^b_l>v'_{j_2}$ is at most $j-j_1$ and the number of $u^b_l$ with $u^b_l>u'_{i_2}$ is at most $i-i_1$ . By Corollary \ref{cor:lrnumber2},
			\begin{equation}\label{eqn:LRK1-3+}
		R(E_{b-2},K_1')\geq R(E_{b-2},||E_a(h_b)||),\quad L(E_{b-2},K_1')\geq L(E_{b-2},||E_a(h_b)||).
			\end{equation}
			By Equations \eqref{eqn:LK1-3}, \eqref{eqn:RK1-3}, and \eqref{eqn:LRK1-3+},
			\begin{eqnarray*}
				L(E_{b-2},K_1)+R(E_{b-2},K_1)&\geq& L(E_{b-2},||E_a(h_b)||)+R(E_{b-2},||E_a(h_b)||)\\
				&& +s+t-(i-i_0)-(j-j_0)\\
					(\text{by Corollary \ref{cor:lrnumber1}})&\geq& wt(E_a(h_b))-wt(E_{b-2}(h_b))+k-n_b+i_0+j_0\\
				(\text{by Equation \ref{eqn:numberl0}})	&>& k-wt(E_{b-2}(h_b)).
			\end{eqnarray*}
			By Lemma \ref{lemma:critgreat}, $K_1$ is not greater than $E_{b-2}$.
		\end{enumerate}
	\end{enumerate}
	\end{proof}

\end{document}